\documentclass[12pt]{amsart}
\usepackage[margin=1.2in]{geometry}
\usepackage{graphicx,latexsym}
\usepackage{comment}
\usepackage{tikz}
\usepackage{amsfonts, amssymb, amsmath, amsthm, bm, hyperref}
\usepackage{tikz}
\usetikzlibrary{matrix,arrows}

\newcommand{\ds}{\displaystyle}
\newcommand{\N}{\mathbb{N}}
\newcommand{\Z}{\mathbb{Z}}
\newcommand{\R}{\mathbb{R}}
\newcommand{\C}{\mathbb{C}}
\newcommand{\Ree}{\operatorname{Re}}
\newcommand{\Imm}{\operatorname{Im}}

\newcommand{\EE}{\mathcal E}
\newcommand{\DD}{\mathcal D}
\newcommand{\SSS}{\mathcal S}

\newcommand{\dx}{{\rm d}x }

\newcommand{\du}{{\rm d}u }

\newcommand{\dt}{{\rm d}t }
\newcommand{\dxi}{{\rm d}\xi }

\newcommand{\csn}{\operatorname{csn}}
\newcommand{\beq}{\begin{eqnarray}}
\newcommand{\eeq}{\end{eqnarray}}

\newcommand{\beqs}{\begin{eqnarray*}}
\newcommand{\eeqs}{\end{eqnarray*}}

\newtheorem{theorem}{Theorem}[section]
\newtheorem{proposition}[theorem]{Proposition}
\newtheorem{lemma}[theorem]{Lemma}
\newtheorem{corollary}[theorem]{Corollary}

\theoremstyle{definition}

\newtheorem{example}[theorem]{Example}

\theoremstyle{remark}
\newtheorem{remark}[theorem]{Remark}

\numberwithin{equation}{section}

\begin{document}
\title[Factorization in Denjoy-Carleman classes]{Factorization in Denjoy-Carleman classes associated to representations of $(\mathbb{R}^{d},+)$}

\author[A. Debrouwere]{Andreas Debrouwere}
\thanks{A. Debrouwere was supported by  FWO-Vlaanderen through the postdoctoral grant 12T0519N}
\address{A. Debrouwere, Department of Mathematics: Analysis, Logic and Discrete Mathematics\\ Ghent University\\ Krijgslaan 281\\ 9000 Gent\\ Belgium}
\email{andreas.debrouwere@UGent.be}

\author[B. Prangoski]{Bojan Prangoski}
\thanks{B. Prangoski was partially supported by the bilateral project ``Microlocal analysis and applications'' funded by the Macedonian and Serbian academies of sciences and arts}
\address{B. Prangoski, Faculty of Mechanical Engineering\\ University Ss. Cyril and Methodius \\ Karpos II bb \\ 1000 Skopje \\ Macedonia}
\email{bprangoski@yahoo.com}

\author[J. Vindas]{Jasson Vindas}
\thanks {J. Vindas was supported by Ghent University through the BOF-grants 01J11615 and 01J04017.}
\address{J. Vindas, Department of Mathematics: Analysis, Logic and Discrete Mathematics\\ Ghent University\\ Krijgslaan 281\\ 9000 Gent\\ Belgium}
\email{jasson.vindas@UGent.be}


\begin{abstract}
For two types of moderate growth representations  of $(\R^d,+)$ on sequentially complete locally convex Hausdorff spaces (including F-representations \cite{G-K-L}), we introduce Denjoy-Carleman classes of ultradifferentiable vectors
and show a strong factorization theorem of Dixmier-Malliavin type for them. In particular, our factorization theorem solves \cite[Conjecture 6.4]{G-K-L} for analytic vectors of representations of $G =(\R^d,+)$. As an application, we show that various convolution algebras and modules of ultradifferentiable functions satisfy the strong factorization property.
\end{abstract}
\maketitle

\section{Introduction}
Let $(\pi,E)$ be a continuous representation of a real Lie group $G$ on a Fr\'echet space $E$ and denote by $E^{\infty}$ the corresponding space of smooth vectors.
The representation $(\pi,E)$ induces an action $\Pi$ of the convolution algebra $\mathcal{D}(G)$ of compactly supported smooth functions on $E$ via
$$
\Pi(f)e=\int_G f(g)\pi(g)e \:{\rm d}g,  \qquad f \in \mathcal{D}(G), e \in E,
$$
which restricts to an action on $E^\infty$. Hence, $E^\infty$ becomes a module over  $(\mathcal{D}(G), \ast)$.  A celebrated result of Dixmier and Malliavin \cite{D-M}  states that $E^\infty$ has the weak factorization property, i.e.,  $E^{\infty}=\mathrm{span}(\Pi(\mathcal{D}(G))E^{\infty})$.

In general, a module $\mathcal{M}$ over a non-unital algebra $\mathcal{A}$ is said to have the \emph{strong (weak) factorization property} if $\mathcal{M} = \mathcal{A} \cdot \mathcal{M} := \{ a \cdot m \, | \, a \in \mathcal{A}, m \in \mathcal{M} \}$ ($\mathcal{M} =  \mathrm{span}(\mathcal{A} \cdot \mathcal{M})$). An algebra $\mathcal{A}$ is said to have the strong or weak factorization property if it has the corresponding property when considered as a module over itself.

Gimperlein et al.~ \cite{G-K-L} showed a variant of the result of Dixmier and Malliavin for analytic vectors: Let $(\pi,E)$ be an $F$-representation of a real Lie group $G$ on a Fr\'echet space $E$.
The corresponding space $E^\omega$ of analytic vectors naturally carries the structure of an $\mathcal{A}(G)$-module, where $\mathcal{A}(G)$ is a suitable convolution algebra of analytic functions having superexponential decay.
They proved that  $E^\omega$ has the weak factorization property  \cite[Theorem 1.1]{G-K-L} and  conjectured that it might be possible to substantially strengthen this result by showing that $E^\omega$  has the strong factorization property  for any $F$-representation \cite[Conjecture 6.4]{G-K-L}. If $(\pi,E)$ is a Banach representation of $(\R^d,+)$, they showed that $E^\omega$  indeed has the strong factorization property \cite[p.\ 679]{G-K-L}; see also \cite{lienau} for the case of bounded Banach representations of  $(\R,+)$.

The main goal of this article is to prove this conjecture for $G = (\R^d, +)$ and further extend it in two directions. On the one hand, we consider two types of moderate growth representations $(\pi,E)$  of $(\R^d,+)$ on sequentially complete
locally convex Hausdorff spaces $E$  (including proto-Banach representations \cite{Glockner} and thus $F$-representations). This setting is indispensable for the applications we have in mind. On the other hand, and more importantly, we consider general Denjoy-Carleman classes defined via a weight sequence $M = (M_p)_{p \in \N}$ \cite{Komatsu} and associated to such representations. In Section \ref{sect-uv} we introduce the spaces $E^{(M)}$ and $E^{\{M\}}$ of ultradifferentiable vectors of Beurling and Roumieu type, respectively, as the space of  vectors in $E$ whose orbit mapping is a (bornological) ultradifferentiable function of Beurling and Roumieu type, respectively.
These naturally carry the structure of a module  over appropriately chosen convolution algebras of ultradifferentiable functions having superexponential decay. Our main result, Theorem \ref{factorization}, asserts that $E^{(M)}$ and $E^{\{M\}}$ satisfy the strong factorization property.
In this framework, the analytic case corresponds to the Roumieu variant of $M = (p!)_{p \in \N}$. However, in general, the sequence $M$ is allowed to grow much slower than $p!$ and we therefore go beyond analyticity, e.g., we are able to treat all  Gevrey sequences $p!^\sigma$, $\sigma > 0$. For Banach representations $(\pi,E)$, the spaces $E^{\{p!^\sigma\}}$, $\sigma \geq 1$, were considered by Goodman and Wallach \cite{Goodman1970,Goodman1971,GW}.

The proof of the factorization theorem consists of two main parts. In Section \ref{Fmultipliers} we study  Fourier multipliers associated to a class of symbols of entire functions on $\C^d$ satisfying certain growth conditions  on tube domains with compact base (with respect to the associated function of $M$). Most importantly, we show how these operators can be properly defined on certain weighed spaces of vector-valued ultradifferentiable functions; the latter spaces naturally arise in a characterization of bounded sets of ultradifferentiable vectors we provide in Section \ref{boundedsetSection}. Next, in Section \ref{thesectiononthefact}, we construct elements $P$ of this entire function symbol class that also satisfy suitable lower bounds, which will allow us to conclude that the Fourier transform of $1/P$ belongs to the algebra involved in the factorization problem. The theory developed in Section \ref{Fmultipliers} enables us to turn the key identity $(1/P) \cdot P = 1$ into a parametrix type identity of convolutor operators that may be applied to the orbits of ultradifferentiable vectors, from which the factorization theorem readily follows. The essential difference with the approach from \cite{G-K-L} (and \cite{D-M}), which allows us to prove  the strong instead of the weak factorization property, is to consider general Fourier multipliers rather than infinite order differential operators.

Closely related to the result of Dixmier and Malliavin is the problem of factorization of convolution algebras of smooth functions, which emerged from Ehrenpreis' work \cite{ehrenpreis} on fundamental solutions of convolution operators. The original question of Ehrenpreis was whether $\mathcal{D}(\R^d)$ has the strong factorization property. Rubel et al.~ \cite{rst} showed that this is not the case for $d\geq 3$. On the other hand, they proved that $\mathcal{D}(\R^d)$ always satisfies the weak factorization property. Dixmier and Malliavin \cite{D-M}  gave a negative answer to this question for $d=2$. Finally, the problem was completely settled by Yulmukhametov \cite{yul} who showed  that for $d = 1$ the space $\mathcal{D}(\R)$ does satisfy  the strong factorization property. Similarly, Miyazaki \cite{Miyazaki}, Petzeltov\'a and Vrbov\'a \cite{pet-vr} and Voigt \cite{Voigt} independently showed that the Schwartz space $\mathcal{S}(\R^d)$ of rapidly decreasing smooth functions has the strong factorization property.

As an application of our main result, we show in Section \ref{sect-application} that various convolution algebras and modules of ultradifferentiable functions satisfy the strong factorization property. Most notably, we prove that the Gelfand-Shilov  spaces $\SSS^{(M)}_{(A)}(\R^d)$ and $\SSS^{\{M\}}_{\{A\}}(\R^d)$ \cite{ppv} have the strong factorization property and that the family of translation invariant spaces $\mathcal{D}^{(M)}_E$  and $\mathcal{D}^{\{M\}}_E$ from \cite{D-P-P-V, D-P-V} factor as $\mathcal{D}^{(M)}_E = \SSS^{(M)}_{(A)}(\R^d) \ast \mathcal{D}^{(M)}_E$ and $\mathcal{D}^{\{M\}}_E = \SSS^{\{M\}}_{\{A\}}(\R^d) \ast \mathcal{D}^{\{M\}}_E$.
Particular instances of the latter spaces are the analogues of the Schwartz spaces $\mathcal{D}_{L^p}$, $1 \leq p < \infty$, and $\dot{\mathcal{B}}$  \cite{SchwartzBook} in the setting of ultradifferentiable functions.

\section{Preliminaries}\label{sect-prel}
Given a lcHs (= locally convex Hausdorff space) $E$, we write $\csn(E)$ for the family of continuous seminorms on $E$ and $\mathfrak{B}(E)$ for the collection of non-empty absolutely convex closed bounded subsets of $E$. For $B \in \mathfrak{B}(E)$ we denote by $E_B$ the subspace of $E$ spanned by $B$ and endowed with the topology generated by the Minkowski functional of $B$. Since $E$ is Hausdorff, the space $E_B$ is normed, while it is complete if $B$ is sequentially complete \cite[Corollary 23.14]{MeiseVogtBook}. In particular,  if $E$ is sequentially complete, $E_B$ is a Banach space for every $B \in \mathfrak{B}(E)$.

All vector-valued integrals in this article will be meant in the weak sense. We will often tacitly use the following fact: Let $E$ be a sequentially complete lcHs. Let ${\bf{f}} \in C(\R^d;E)$ be such that for all $p \in \csn(E)$  it holds that
$\int_{\R^d} p({\bf{f}}(x)) \dx < \infty$.
Then, the $E$-valued  weak integral $\int_{\R^d} {\bf{f}}(x) \dx$ exists.

Next, we introduce weight sequences in order to discuss ultradifferentiability. A sequence $M = (M_{p})_{p \in \N}$ of positive numbers is called a \emph{weight sequence} if  $M_0 = M_1 = 1$, $\lim_{p \to \infty} M_{p}^{1/p} =  \infty$, and $M$ is log-convex, i.e., $M_{p}^{2} \leq M_{p-1} M_{p+1}$ for all $p \in \Z_+$. We write $M_\alpha = M_{|\alpha|}$ for $\alpha \in \Z^d$. Furthermore, we set $m_p = M_p/M_{p-1}$ for $p \in \Z_+$. We shall make use of the following two conditions on weight sequences $M$:
\begin{itemize}
\item[$(M.2)\ $]  $\ds M_{p+q} \leq C_0H^{p+q} M_{p} M_{q}$, $p,q\in \N$, for some $C_0,H\geq1$;
\item[$(M.2)^*$] $2m_p \leq m_{Np}$, $p \geq p_0$, for some $p_0,N \in \Z_+$.
\end{itemize}
Condition $(M.2)$ is also known as `stability under ultradifferential operators' and was introduced by Komatsu \cite{Komatsu}. Condition $(M.2)^*$ was introduced by Bonet et al. \cite{BMM} without a name; we use here the same notation as in \cite{D-V}. We refer to \cite{JGSS} for more information and various equivalent formulations of these conditions; in particular, one might verify that $(M.2)^*$ is equivalent to the condition $(M.5)$ employed in \cite{ppv}. The most important examples of weight sequences that satisfy the conditions $(M.2)$ and $(M.2)^*$ are the \emph{Gevrey sequences} $p!^{\sigma}$, $\sigma>0$.

The \emph{associated function} of $M$ is defined as
$$
\nu_M(t) = \sup_{p\in\N}\log \frac{t^p}{M_p},\qquad t \geq 0.
$$
We have that $\nu_{p!^\sigma}(t) \asymp t^{1/\sigma}$, $\sigma >0$. We extend $\nu_M$ to $\R^d$ as the radial function $\nu_M(x) = \nu_M(|x|)$, $x \in \R^d$. The weight sequence $M$ satisfies $(M.2)$ if and only if \cite[Proposition 3.6]{Komatsu}
 $$
2\nu_M(t) \leq \nu_M(Ht) + \log C_0, \qquad t \geq 0,
$$
where $C_0$ and $H$ are the constants occurring in $(M.2)$. If $M$ satisfies $(M.2)^*$, then \cite[Lemma 12]{BMM}
$$
\nu_M(2t) \leq L\nu_M(t) + \log C, \qquad t \geq 0,
$$
for some $C,L > 0$. If $M$ satisfies $(M.2)$, then the converse also holds true \cite[Proposition 13]{BMM}.

The relation $M \subset N$ between two weight sequences means that there are $C,L \geq 1$ such that
$M_p\leq CL^{p}N_{p}$ for all $p\in\N$. By \cite[Lemma 3.8]{Komatsu}, we have that $M \subset N$ if and only if
$$
\nu_N(t) \leq \nu_M(Lt) + \log C, \qquad t \geq 0.
$$
for some $C,L > 0$.

We write $\mathfrak{R}$ for the set of all non-decreasing sequences $r = (r_j)_{j \in \N}$  of positive numbers such that $r_0 = r_1 = 1$ and $r_j \to \infty$ as $j \to \infty$. It is partially ordered and directed by the relation $r \preceq s$  which means that $r_j\leq s_j$ for all $j\in \N$. This set of sequences plays a fundamental role in Komatsu's approach to spaces of vector-valued ultradifferentiable functions of Roumieu type \cite{Komatsu3}. The following lemma is a simple but very useful tool and hints the way in which $\mathfrak{R}$ often appears involved in some arguments.
\begin{lemma}[{\cite[Lemma 3.4]{Komatsu3}}]
\label{klemma}
Let $(a_p)_{p \in \N}$ be a sequence of positive numbers.
\begin{itemize}
\item[$(i)$]
$
\ds\sup_{p \in \N} \frac{a_p}{h^p} < \infty
$
for some $h > 0$ if and only if
$
\ds\sup_{p \in \N} \frac{a_p}{\prod_{j = 0}^p r_j} < \infty
$
for all $r \in \mathfrak{R}$.
\item[$(ii)$]
$
\ds\sup_{p \in \N} h^pa_p < \infty
$
for all $h > 0$ if and only if
$
\ds \sup_{p \in \N} a_p\prod_{j = 0}^p r_j < \infty
$
for some $r \in \mathfrak{R}$.
\end{itemize}
\end{lemma}

We now introduce two basic spaces of vector-valued ultradifferentiable functions.  Let $M$ be a weight sequence and let $E$ be a lcHs. Following Komatsu \cite{Komatsu3}, we define the space  $\EE^{(M)}(\R^d;E)$ of \emph{$E$-valued ultradifferentiable functions of class $(M)$ (of Beurling type)} as the space consisting of all $\bm{\varphi}\in C^{\infty}(\R^d;E)$ such that for all $K\Subset \R^d$, $h>0$, and $p\in \csn(E)$ it holds that
\beqs
p_{h,K}(\bm{\varphi}):= \sup_{\alpha\in\N^d}\sup_{x\in K}\frac{h^{|\alpha|}p(\partial^{\alpha}\bm{\varphi}(x))}{M_{\alpha}}<\infty.
\eeqs
We endow $\EE^{(M)}(\R^d;E)$  with the locally convex topology generated by the system of seminorms $\{p_{h,K} \, | \,  h>0, K\Subset \R^d, p\in\csn(E)\}$. Similarly,  we define the space  $\EE^{\{M\}}(\R^d;E)$ of \emph{$E$-valued ultradifferentiable functions of class $\{M\}$ (of Roumieu type)} as the space consisting of all $\bm{\varphi}\in C^{\infty}(\R^d;E)$ such that for all $K\Subset \R^d$, $r\in\mathfrak{R}$, and $p\in \csn(E)$ it holds that
\beqs
p_{r,K}(\bm{\varphi}):=\sup_{\alpha\in\N^d}\sup_{x\in K}\frac{p(\partial^{\alpha}\bm{\varphi}(x))}{M_{\alpha} \prod_{j = 0}^{|\alpha|} r_j}<\infty.
\eeqs
The space $\EE^{\{M\}}(\R^d;E)$ is endowed with the locally convex topology generated by the system of seminorms  $\{p_{r,K}\, | \, r\in\mathfrak{R}, K\Subset \R^d, p\in\csn(E)\}$. We shall employ $[M]$ as a common notation for $(M)$ and $\{M\}$. In addition, we shall often first state assertions for the Beurling case followed in parenthesis by the corresponding statements for the Roumieu case.

\begin{remark} The space of complex-valued ultradifferentiable functions of Roumieu type on $\R^d$ is usually defined as
$$
\mathcal{E}^{\{M\}}(\R^d) = \varprojlim_{K \Subset \R^d} \varinjlim_{h \to 0^+} \mathcal{E}^{M_p,h}(K),
$$
where $K$ runs over all regular compact subsets of $\R^d$ and $\mathcal{E}^{M_p,h}(K)$ denotes the Banach space consisting of all $\varphi \in C^\infty(K)$ such that
\beqs
\sup_{\alpha\in\N^d}\sup_{x\in K}\frac{h^{|\alpha|}|\partial^{\alpha}\varphi(x)|}{M_{\alpha}}<\infty.
\eeqs
Lemma \ref{klemma} implies that the spaces $\EE^{\{M\}}(\R^d;\C)$ and $\EE^{\{M\}}(\R^d)$ agree as sets. If $M$ satisfies Komatsu's condition $(M.2)'$ \cite{Komatsu} (which is a weaker version of $(M.2)$), these spaces also coincide topologically \cite[Corollary 1]{D-P-V-proj}.
\end{remark}

\section{Ultradifferentiable vectors and the factorization theorem}\label{sect-uv}
Given a lcHs E,  we denote by $\operatorname{GL}(E)$ its group of isomorphisms. By a \emph{representation of $(\R^d,+)$ on $E$} we mean a group homomorphism $\pi: (\R^d,+) \rightarrow \operatorname{GL}(E)$ such that the mapping
$$
\R^d \times E \rightarrow E, \quad  (x,e) \mapsto \pi(x)e
$$
is separately continuous. We denote by $E^\infty$ the space of smooth vectors in $E$, i.e., the space consisting of all $e \in E$ whose orbit mapping
$$
\gamma_e: \R^d \rightarrow E, \quad \gamma_e(x) = \pi(x)e,
$$
belongs to $C^\infty(\R^d;E)$. We will employ the following short-hand notation
$$
e_\alpha = \partial^{\alpha}\gamma_e(0), \qquad e \in E^\infty,\, \alpha \in \N^d.
$$
Note that $\partial^{\alpha}\gamma_e = \gamma_{e_\alpha}$.

If $E$ is a Banach space, $(\pi,E)$ is called a \emph{Banach representation}.  In such a case, there are $\kappa, C > 0$ such that
$$
\| \pi(x)e\|_E \leq Ce^{\kappa|x|}\|e\|_E, \qquad  x \in \R^d, e \in E.
$$
Motivated by the above inequality, we now introduce two new classes of representations on general lcHs $E$. A representation $(\pi,E)$ is said to be a \emph{projective generalized proto-Banach representation} if
\begin{equation}
\label{projective}
\begin{gathered}
\forall p \in \csn(E) \, \exists q_p \in \csn(E) \, \exists \kappa_p > 0 \, \forall x \in \R^d \, \forall e \in E \, : \, \\
p(\pi(x)e) \leq e^{\kappa_p|x|}q_p(e)
\end{gathered}
\end{equation}
and an \emph{inductive generalized proto-Banach representation} if
\begin{equation}
\label{inductive}
\begin{gathered}
\forall B \in \mathfrak{B}(E) \, \exists A_B \in \mathfrak{B}(E) \, \exists \kappa_B > 0 \, \forall x \in \R^d \, \forall e \in E_B \, : \, \\
\|\pi(x)e\|_{E_{A_B}} \leq e^{\kappa_B|x|}\|e\|_{E_B}.
\end{gathered}
\end{equation}
Every Banach representation is both a projective and an inductive generalized proto-Banach representation. Furthermore, every proto-Banach representation \cite{Glockner} (and thus particularly every $F$-representation \cite{G-K-L}) is a
projective generalized proto-Banach representation.

Let $(\pi,E)$ be a representation. Let $M$ be a weight sequence. A vector $e \in E$ is called an \emph{ultradifferentiable vector of class $[M]$  in $E$} if the orbit mapping $\gamma_e$ is bornologically ultradifferentiable of class $[M]$ \cite{Komatsu3}, i.e.,  if there is $B \in \mathfrak{B}(E)$ such that $\gamma_e \in  \mathcal{E}^{[M]}(\R^d;E_B)$. The space consisting of all ultradifferentiable vectors of class $[M]$ in $E$ is denoted by $E^{[M]}$. We endow $E^{[M]}$ with a convex vector bornology in the following way: A set $\widetilde{B} \subset E^{[M]}$ is defined to be bounded if there is $B \in \mathfrak{B}(E)$ such that $\{\gamma_e  \, | \, e \in  \widetilde{B} \}$ is contained and bounded in $\mathcal{E}^{[M]}(\R^d; E_B)$. 
\begin{remark}
If $(\pi,E)$ is an $F$-representation, the space $E^\omega$ of analytic vectors  considered in \cite{G-K-L} coincides with $E^{\{p!\}}$, as follows from \cite[Proposition 12]{B-D} and the remark added at the end of that article.
\end{remark}

Let us introduce the necessary concepts to state our strong factorization theorem for ultradifferentiable vectors. Assume that $E$ is sequentially complete. We define the Fr\'echet space
$$
C_{\operatorname{exp}}(\R^d)= \{ f \in C(\R^d) \, | \, \sup_{x \in \R^d} |f(x)|e^{n|x|} < \infty, \quad \forall n \in \N \}.
$$
Note that $(C_{\operatorname{exp}}(\R^d), \ast)$ is a Fr\'echet algebra. If $(\pi,E)$ is a projective or an inductive generalized proto-Banach representation, we set
$$
\Pi(f)e = \int_{\R^d} f(x) \pi(x)e\:\dx = \int_{\R^d} f(x) \gamma_e(x)\dx \in E, \qquad f \in  C_{\operatorname{exp}}(\R^d),\, e \in E.
$$
When $(\pi,E)$ is a projective generalized proto-Banach representation, the mapping
\beqs
C_{\operatorname{exp}}(\R^d)\times E\rightarrow E: \ (f,e)\mapsto \Pi(f)e 
\eeqs
is continuous. If $(\pi,E)$ is an inductive generalized proto-Banach representation, then $\Pi(f)e \in E_{A_B}$ for all $f \in  C_{\operatorname{exp}}(\R^d)$ and $e \in E_B$, and the mapping
\beqs
C_{\operatorname{exp}}(\R^d)\times E_B\rightarrow E_{A_B}:\ (f,e)\mapsto \Pi(f)e
\eeqs
is continuous. In particular, for each $f\in C_{\operatorname{exp}}(\R^d)$ and each bounded subset $B$ of $E$, $\Pi(f)B$ is bounded in $E$.


For $h > 0$ we define the Fr\'echet space
$$
\mathcal{K}^{M,h}(\R^d) = \{ \varphi \in C^\infty(\R^d) \, | \, \sup_{\alpha \in \N^d} \sup_{x \in \R^d} \frac{h^{|\alpha|}|\partial^{\alpha}\varphi(x)|e^{n|x|}}{M_\alpha} < \infty, \quad \forall n \in \N   \}.
$$
We set
$$
\mathcal{K}^{(M)}(\R^d) = \varprojlim_{h \to \infty} \mathcal{K}^{M,h}(\R^d) \qquad \mbox{and}\qquad \mathcal{K}^{\{M\}}(\R^d) = \varinjlim_{h \to 0^+} \mathcal{K}^{M,h}(\R^d).
$$
We remark that $\mathcal{K}^{[M]}(\R^d)$ is the space of ultradifferentiable vectors of class $[M]$ in $C_{\operatorname{exp}}(\R^d)$ under the regular representation (cf.\ Subsection \ref{weighted}).

The next theorem is the main result of this article.
\begin{theorem}\label{factorization}
Let $(\pi,E)$ be either a projective or an inductive generalized proto-Banach representation of $(\R^d,+)$ on a sequentially complete lcHs $E$. Let $M$ be a weight sequence satisfying $(M.2)$ and $(M.2)^*$. Then,
$$
\Pi(\mathcal{K}^{[M]}(\R^d)) E = \Pi(\mathcal{K}^{[M]}(\R^d)) E^{[M]} = E^{[M]}.
$$
Moreover, for every bounded set $\widetilde B \subset E^{[M]}$ there is $\psi \in \mathcal{K}^{[M]}(\R^d)$ and a bounded set $\widetilde A \subseteq E^{[M]}$ such that $\Pi(\psi)\widetilde A= \widetilde B$.
\end{theorem}
The proof of Theorem \ref{factorization} is postponed to Section \ref{thesectiononthefact}. In preparation for it, we need to establish a number of results in the next two sections.
\section{Bounded subsets of ultradifferentiable vectors}
\label{boundedsetSection}
We provide in this section a characterization of the bounded subsets of $E^{[M]}$ for projective and inductive generalized proto-Banach representations of $(\R^{d},+)$ (for which we use the  same notation as in \eqref{projective} and \eqref{inductive}). The following spaces of $E$-valued ultradifferentiable functions are involved in such a characterization.

Let $(\pi,E)$ be a representation on the lcHs $E$ and fix a weight sequence $M$. Let $\bm{\kappa} = (\kappa_p)_{p \in \csn(E)}$ be a net of positive numbers (the set $\csn(E)$ is directed by $p\leq q$, which means $p(e)\leq q(e)$, $\forall e\in E$).  For $h > 0$ we define  $\mathcal{Q}^{M,h}_{{\boldsymbol{\kappa}}}(\R^d;E)$ as the space consisting of all ${\bf{f}} \in C^\infty(\R^d;E)$ such that for all $p \in \csn(E)$ it holds that
$$
p_{\bm{\kappa},h}({\bf{f}}) =  \sup_{\alpha \in \N^d} \sup_{x \in \R^d} \frac{h^{|\alpha|}p(\partial^{\alpha}{\bf{f}}(x))e^{-\kappa_p|x|}}{M_\alpha} < \infty.
$$
We endow $\mathcal{Q}^{M,h}_{{\boldsymbol{\kappa}}}(\R^d;E)$  with the locally convex topology generated by the system of seminorms  $\{p_{\bm{\kappa},h} \, | \, p\in\csn(E)\}$. Similarly, for $r \in \mathfrak{R}$ we define $\mathcal{Q}^{M,r}_{{\boldsymbol{\kappa}}}(\R^d;E)$ as the space consisting of all ${\bf{f}} \in C^\infty(\R^d;E)$ such that for all $p \in \csn(E)$ it holds that
$$
p_{\bm{\kappa},r}({\bf{f}}) =  \sup_{\alpha \in \N^d} \sup_{x \in \R^d} \frac{\prod_{j = 0}^{|\alpha|} r_jp(\partial^{\alpha}{\bf{f}}(x))e^{-\kappa_p|x|}}{M_{\alpha}} < \infty,
$$
and we endow it  with the locally convex topology generated by the system of seminorms  $\{p_{\bm{\kappa},r} \, | \, p\in\csn(E)\}$.

\begin{proposition}\label{char-E*} Let $(\pi,E)$ be either a projective or an inductive generalized proto-Banach representation of $(\R^d,+)$ on a lcHs $E$.
Let $\widetilde{B} \subseteq E^\infty$. The following statements are equivalent:
\begin{itemize}
\item[$(i)$] $\widetilde{B}$ is a bounded subset of $E^{[M]}$.
\item[$(ii)$] There are $B \in \mathfrak{B}(E)$ and $r \in \mathfrak{R}$  ($h > 0$) such that $e_\alpha \in E_B$ for all $\alpha \in \N^d$ and $e\in\widetilde{B}$ and
$$
\sup_{e\in\widetilde{B}}\sup_{\alpha \in \N^d} \frac{\prod_{j=0}^{|\alpha|}r_j\|e_\alpha\|_{E_B}}{M_\alpha} < \infty \qquad \left(\sup_{e\in\widetilde{B}}\sup_{\alpha \in \N^d} \frac{h^{|\alpha|}\|e_\alpha\|_{E_B}}{M_\alpha} < \infty \right).
$$
\item[$(iii)$] There is $r \in \mathfrak{R}$ ($h>0$) such that for all $p \in \operatorname{csn}(E)$ it holds that
$$
\sup_{e\in\widetilde{B}}\sup_{\alpha \in \N^d} \frac{\prod_{j=0}^{|\alpha|}r_jp(e_\alpha)}{M_\alpha} < \infty \qquad  \left(\sup_{e\in\widetilde{B}}\sup_{\alpha \in \N^d} \frac{h^{|\alpha|}p(e_\alpha)}{M_\alpha} < \infty\right).
$$
\item[$(iv)$] \emph{($(\pi,E)$ projective generalized proto-Banach representation)}  There is $r \in \mathfrak{R}$ ($h>0$)  such that $\mathbf{B}=\{\gamma_e \, | \, e\in\widetilde{B}\}$ is a bounded subset of  $\mathcal{Q}^{M,r}_{{\boldsymbol{\kappa}}}(\R^d;E)$ ($\mathcal{Q}^{M,h}_{{\boldsymbol{\kappa}}}(\R^d;E)$).

\medskip

\noindent \emph{($(\pi,E)$ inductive generalized proto-Banach representation)}  There are $B \in \mathfrak{B}(E)$ and $r \in \mathfrak{R}$  ($h > 0$)  such that $\mathbf{B}=\{\gamma_e \, | \, e\in\widetilde{B}\}$ is a bounded subset of  $\mathcal{Q}^{M,r}_{\kappa_B}(\R^d;E_{A_B})$ ($\mathcal{Q}^{M,h}_{\kappa_B}(\R^d;E_{A_B})$).
\end{itemize}
\end{proposition}
The ensuing lemma will be employed in the proof of Proposition \ref{char-E*}.
\begin{lemma}\label{smooth-E-B}
Let ${\bf{f}} \in C^\infty(\R^d;E)$ and suppose that there is  $B \in \mathfrak{B}(E)$ such that for every $K \Subset \R^d$ and $N \in \N$ there is $R > 0$ such that
$$
\left \{ \partial^{\alpha}{\bf{f}}(x) \, | \,  x \in K,\, |\alpha| \leq N \right\} \subseteq R B.
$$
Then, ${\bf{f}} \in C^\infty(\R^d;E_B)$.
\end{lemma}

\begin{proof}
For $e' \in E'$ we write $f_{e'} = \langle e', {\bf{f}}\rangle \in C^\infty(\R^d)$. Let $x_0 \in \R^d$ and $\alpha \in \N^d$ be arbitrary. By applying the second-order Taylor theorem to $\partial^\alpha f_{e'}$, we obtain that
\begin{multline*}
\frac{1}{|x-x_0|}\left|\partial^{\alpha}f_{e'}(x) - \partial^{\alpha}f_{e'}(x_0) - \sum_{j=1}^d (x_j -x_{0,j})\partial^{\alpha+e_j}f_{e'}(x_0) \right | \\
\leq d^2|x-x_0| \max_{|\beta| = |\alpha| + 2} \sup_{|y-x_0| \leq 1}|\partial^{\beta}f_{e'}(y)|
\end{multline*}
for all $0 < |x - x_0| \leq 1$ and $e' \in E'$. Choose $R > 0$ such that
$$
\left \{ \partial^{\beta}{\bf{f}}(y) \, | \, |x_0 - y| \leq 1,\, |\beta| = |\alpha| + 2 \right\} \subseteq R B.
$$
Since $\partial^{\alpha}f_{e'}= \langle e', \partial^{\alpha}{\bf{f}}\rangle$, we infer that
$$
\left |\left \langle e', \frac{1}{|x-x_0|}\left(\partial^{\alpha}{\bf{f}}(x) - \partial^{\alpha}{\bf{f}}(x_0) - \sum_{j=1}^d (x_j -x_{0,j})\partial^{\alpha+e_j}{\bf{f}}(x_0)\right) \right \rangle \right|  \leq d^2R|x-x_0|
$$
for all $e' \in B^\circ$. The bipolar theorem implies that
$$
\frac{1}{|x-x_0|}\left \| \partial^{\alpha}{\bf{f}}(x) - \partial^{\alpha}{\bf{f}}(x_0) - \sum_{j=1}^d (x_j -x_{0,j})\partial^{\alpha+e_j}{\bf{f}}(x_0) \right \|_{E_B} \leq  d^2R|x-x_0|,
$$
whence ${\bf{f}} \in C^\infty(\R^d;E_B)$.
\end{proof}

\begin{proof}[Proof of Proposition \ref{char-E*}] $(i) \Rightarrow (ii)$ This follows from Lemma \ref{klemma}.

 $(ii) \Leftrightarrow (iii)$ Obvious.

 For the remaining equivalences we distinguish two cases. Suppose first that $(\pi,E)$ is a projective generalized proto-Banach representation.

 $(iii) \Rightarrow (iv)$ Obvious.

 $(iv) \Rightarrow (i)$ We only show the Beurling case as the Roumieu case is similar. The set
$$
B' = \left \{ \frac{\prod_{j=0}^{|\alpha|}r_j\partial^{\alpha}\gamma_e(x)e^{-|x|^2}}{M_\alpha} \, | \, x \in \R^d,\, \alpha \in \N^d, e\in\widetilde{B}  \right\}
$$
is bounded in $E$. Let $B$ be the closed absolutely convex hull of $B'$. Lemma \ref{smooth-E-B} yields that $\mathbf{B} \subset C^\infty(\R^d;E_B)$. It is then clear from the definition of $B'$ that $\{\gamma_e  \, | \, e \in  \widetilde{B} \}$ is contained and bounded in $\mathcal{E}^{(M)}(\R^d; E_B)$.

Next, assume that $(\pi,E)$ is an inductive  generalized proto-Banach representation.

 $(ii) \Rightarrow (iv)$ This follows from Lemma \ref{smooth-E-B}.

 $(iv) \Rightarrow (i)$ Obvious.
\end{proof}

We end this section with two remarks.

\begin{remark}\label{char-E*-Frechet}
Let $E$ be a  Fr\'echet space. In the Beurling case, the conditions from Proposition \ref{char-E*} are also equivalent to: For all $h> 0$ and $p \in \csn(E)$ it holds that
$$
\sup_{e\in\widetilde{B}}\sup_{\alpha \in \N^d} \frac{h^{|\alpha|}p(e_\alpha)}{M_\alpha} < \infty.
$$
The above condition is equivalent to condition $(iii)$ from Proposition \ref{char-E*} because of Lemma \ref{klemma} and the following result.
\end{remark}
\begin{lemma}
Let $r^{(n)} \in \mathfrak{R}$ for all $n \in \N$. There is $r \in \mathfrak{R}$ such that for every $n\in\N$ there is $k \in\Z_+$ such that $r_j \leq r^{(n)}_j$ for all $j\geq k$.
\end{lemma}
\begin{proof}
We inductively define a sequence $(j_k)_{k \in \Z_+}$ of natural numbers with $j_1 = 1$ satisfying
$
j_{k+1} > j_k,$ $r^{(k+1)}_{j_{k+1}} \geq r^{(k)}_{{j_k}},$ and $\min_{0 \leq l \leq k+1}r^{(l)}_{j_{k+1}} \geq k+1,$
for all $k \in \Z_+$. We set $r_j = 1$ for $j=0,\ldots,j_2-1$, and
$
r_j = \min_{0 \leq l \leq k} r^{(l)}_{j_k},$ if  $j_k \leq j < j_{k+1}$ for some $k\geq 2$.
Then, $r$ belongs to $\mathfrak{R}$ and satisfies the conclusion of the lemma.
\end{proof}
\begin{remark}\label{Banach-top}
If $(\pi,E)$ is a Banach representation, the space $E^{[M]}$ may be naturally endowed with a locally convex topology in the following way (cf.\ Proposition \ref{char-E*} and Remark \ref{char-E*-Frechet}):
$$
E^{(M)} = \varprojlim_{h \to \infty} E^{M,h} \qquad \mbox{and} \qquad E^{\{M\}} = \varinjlim_{h \to 0^+}  E^{M,h},
$$
where $E^{M,h}$ denotes the Banach space consisting of all $e \in E^\infty$ such that
$$
 \sup_{\alpha \in \N^d} \frac{h^{|\alpha|}\|e_\alpha\|_{E}}{M_\alpha} < \infty.
$$
We denote this topology by $\tau$. By using a similar technique as in the proof of \cite[Proposition 5.1]{D-P-P-V}, it can be shown that the $(LB)$-space $(E^{\{M\}}, \tau)$ is boundedly retractive, i.e., for every bounded set $\widetilde{B} \subset (E^{\{M\}}, \tau)$ there is $h>0$ such that $\widetilde{B}$ is a bounded subset of $E^{M,h}$ and the topology induced on $\widetilde{B}$ by $E^{M,h}$ coincides with the one induced by $\tau$. In particular, the bornology induced by $\tau$ on $E^{\{M\}}$ coincides with the original bornology defined on $E^{\{M\}}$. Similarly, for every bounded set $\widetilde{B} \subset E^{(M)}$  there is $r\in\mathfrak{R}$  such that, for  $N = (M_p/\prod_{j=0}^pr_j)_{p \in \N}$, $\widetilde{B}$ is a bounded subset of $E^{N,1}$ and the topology induced on $\widetilde{B}$ by $E^{N,1}$ coincides with the one induced by $\tau$. In particular, the bornology induced by $\tau$ on $E^{(M)}$ coincides with the original bornology defined on $E^{(M)}$.
\end{remark}

\section{On a class of Fourier multipliers}\label{Fmultipliers}
In this section we discuss the space of Fourier multipliers associated to  $\mathcal{K}^{(M)}(\R^d)$  and show how these operators may be defined on  the spaces $\mathcal{Q}^{M,h}_{{\boldsymbol{\kappa}}}(\R^d;E)$ from the previous section by suitably interpreting them as convolution operators.

 Throughout this section, we fix a sequentially complete lcHs $E$, a net $\boldsymbol \kappa = (\kappa_p)_{p \in \csn(E)}$ of positive numbers and a weight sequence $M$ satisfying $(M.2)$.  We denote by $C_0,H$  the constants occurring in the latter condition. In addition, we always assume in this section that $\mathcal{K}^{(M)}(\mathbb{R}^{d})$ is non-trivial\footnote{The non-triviality of $\mathcal{K}^{(M)}(\mathbb{R}^{d})$ can be characterized in terms of the behavior of the weight sequence in a precise fashion. Indeed \cite[Proposition 2.7, Proposition 4.4, and Theorem 5.9]{D-V2018},  under condition $(M.2)$, we have $\mathcal{K}^{(M)}(\mathbb{R}^{d})\neq \{0\}$ if and only if $\lim_{p\to\infty} m_p/\log p=\infty$. The latter certainly holds if the sequence satisfies $(M.2)^*$ because it implies the stronger relation $p^{\sigma}=O(m_p)$ for some $\sigma>0$. }. Given $L > 0$, we consider the tube domain  $V_L = \R^d + i B(0,L)$.

 For $h > 0$  we define the Fr\'echet space
$$
\mathcal{U}^{M,h}(\C^d) = \{ \varphi \in \mathcal{O}(\C^d) \, | \,  \sup_{z \in  V_n} |\varphi(z)|e^{\nu_M( h\operatorname{Re} z)} < \infty, \quad \forall n \in \N \}.
$$
We set
$$
\mathcal{U}^{(M)}(\C^d) = \varprojlim_{h \to \infty} \mathcal{U}^{M,h}(\C^d)\qquad \mbox{and} \qquad \mathcal{U}^{\{M\}}(\C^d) = \varinjlim_{h \to 0^+} \mathcal{U}^{M,h}(\C^d).
$$
The Fourier transform is a topological isomorphism  from $\mathcal{K}^{[M]}(\R^d)$ onto $\mathcal{U}^{[M]}(\C^d)$, where we fix the constants in the Fourier transform as follows
$$
\mathcal{F}(\varphi)(\xi) = \widehat{\varphi}(\xi) = \int_{\R^d} \varphi(x) e^{-2\pi i x \cdot \xi} \dx.
$$

We now introduce spaces of multipliers and convolutors associated to $\mathcal{U}^{(M)}(\C^d)$ and $\mathcal{K}^{(M)}(\R^d)$, respectively.  For $h > 0$ we define the Fr\'echet space
$$
\mathcal{P}^{M,h}(\C^d) = \{ \varphi \in \mathcal{O}(\C^d) \, | \, \sup_{z \in V_n}  |\varphi(z)|e^{-\nu_M( h\operatorname{Re} z)} < \infty, \quad \forall n \in \N \}.
$$
For $P \in \mathcal{P}^{M,h}(\C^d)$ fixed, the multiplication operator $\mathcal{U}^{(M)}(\C^d) \rightarrow \mathcal{U}^{(M)}(\C^d): \varphi \mapsto P \cdot \varphi$ is  continuous.  Next, for $h > 0$ we define the $(LB)$-space
$$
\mathcal{O}^{M,h}_C(\R^d) = \varinjlim_{n \in \N} \mathcal{O}^{M,h,n}_C(\R^d),
$$
where $\mathcal{O}^{M,h,n}_C(\R^d)$ stands for  the Banach space consisting of all $\varphi \in C^\infty(\R^d)$ such that
$$
 \sup_{\alpha \in \N^d} \sup_{x \in \R^d} \frac{h^{|\alpha|}|\partial^\alpha\varphi(x)|e^{-n|x|}}{M_\alpha} < \infty.
$$
We denote by $\mathcal{O}'^{M,h}_C(\R^d)$ the strong dual of $\mathcal{O}^{M,h}_C(\R^d)$. For $g \in \mathcal{O}'^{M,h}_C(\R^d)$ fixed, the convolution operator
$$
\mathcal{K}^{(M)}(\R^d) \rightarrow \mathcal{K}^{(M)}(\R^d): \varphi \mapsto g \ast \varphi, \qquad (g \ast \varphi) (x) = \langle g(t), \varphi(x-t) \rangle
$$
is continuous,
where the duality on the right is in the sense of $\langle \mathcal{O}'^{M,h}_C(\R^d), \mathcal{O}^{M,h}_C(\R^d)\rangle$.
 For each $h>0$, $\mathcal{K}^{\{M\}}(\R^d)$ can be viewed as a subspace of $\mathcal{O}'^{M,h}_C(\R^d)$ via the linear injection
\beq\label{inc-k-o}
\mathcal{K}^{\{M\}}(\R^d)\rightarrow \mathcal{O}'^{M,h}_C(\R^d):\,\, \varphi\mapsto \left(\psi\mapsto\int_{\R^d}\varphi(x)\psi(x)\dx\right).
\eeq

We also need spaces of $E$-valued ultradistributions. Set
$$
\mathcal{K}'^{(M)}(\R^d;E) = \mathcal{L}(\mathcal{K}^{(M)}(\R^d), E) \qquad \mbox{and} \qquad \mathcal{U}'^{(M)}(\C^d;E) = \mathcal{L}(\mathcal{U}^{(M)}(\C^d), E).
$$
We define the Fourier transform from $\mathcal{U}'^{(M)}(\C^d;E)$ onto $\mathcal{K}'^{(M)}(\R^d;E)$ via duality. Let $h> 0$. For $P \in \mathcal{P}^{M,h}(\C^d)$ fixed, we define
$$
\langle P \cdot {\bf{f}}, \varphi \rangle = \langle {\bf{f}}, P \cdot \varphi \rangle, \qquad  \varphi \in \mathcal{U}^{(M)}(\C^d),
 $$
for  ${\bf{f}} \in \mathcal{U}'^{(M)}(\C^d;E)$. Similarly, for $g \in \mathcal{O}'^{M,h}_C(\R^d)$ fixed, we define
 $$
\langle g \ast {\bf{f}}, \varphi \rangle = \langle {\bf{f}}, (g \ast \check{\varphi})\check{\phantom{n}} \rangle, \qquad \varphi \in \mathcal{K}^{(M)}(\R^d),
 $$
 for ${\bf{f}} \in \mathcal{K}'^{(M)}(\R^d;E)$.
\begin{proposition}\label{quantization}
For each $h > 0$ there exists a continuous linear mapping $\widetilde{\mathcal{F}}_h :  \mathcal{P}^{M,h}(\C^d) \rightarrow \mathcal{O}'^{M,2\sqrt{d}Hh/\pi}_C(\R^d)$ such that
$$
\mathcal{F}(P \cdot {\bf{f}}) = \widetilde{\mathcal{F}}_h(P) \ast \mathcal{F}({\bf{f}}), \qquad P \in \mathcal{P}^{M,h}(\C^d), {\bf{f}} \in \mathcal{U}'^{(M)}(\C^d;E).
$$
Moreover, for all $0 < k < h$,
\begin{center}
\begin{tikzpicture}
  \matrix (m) [matrix of math nodes,row sep={2em}, column sep={5em}]
  {
  \mathcal{U}^{\{M\}}(\C^d) & \mathcal{P}^{M,k}(\C^d) & \mathcal{P}^{M,h}(\C^d)\\
   & \mathcal{O}'^{M,2\sqrt{d}Hk/\pi}_C(\R^d) & \\
   \mathcal{K}^{\{M\}}(\R^d) & & \mathcal{O}'^{M,2\sqrt{d}Hh/\pi}(\R^d)\\
   };
   \draw[->] (m-1-1) -- (m-1-2) node[midway,above]{can. inclusion};
   \draw[->] (m-1-2) -- (m-1-3) node[midway,above]{can. inclusion};
   \draw[->] (m-1-2) -- (m-2-2) node[midway,left]{$\widetilde{\mathcal{F}}_k$};
   \draw[->] (m-1-3) -- (m-3-3) node[midway,right]{$\widetilde{\mathcal{F}}_h$};
  \draw[->] (m-1-1) -- (m-3-1) node[midway,left]{$\mathcal{F}$};
  \draw[->] (m-3-1) -- (m-2-2) node[midway, above]{inc. \eqref{inc-k-o}${}\quad\quad$};
  \draw[->] (m-3-1) -- (m-3-3) node[midway, above]{inc. \eqref{inc-k-o}};
  \draw[->] (m-2-2) -- (m-3-3) node[near start,above]{${}\quad\quad\quad$ transpose of} node[near end, above]{$\quad\quad$ can. incl.};
\end{tikzpicture}
\end{center}
is a commutative diagram of continuous maps.
\end{proposition}

The proof of Proposition \ref{quantization} is based on the mapping properties of the short-time Fourier transform \cite{Grochenig}. The translation and modulation operators are denoted by $T_xf(t) = f(t- x)$ and $M_\xi f(t) = e^{2\pi i \xi \cdot t} f(t)$, $x, \xi \in \R^d$.  The \emph{short-time Fourier transform} (STFT) of a function $f \in L^2(\R^d)$ with respect to a window function $\psi \in L^2(\R^d)$ is defined as
\begin{equation}
V_\psi f(x,\xi) = (f, M_\xi T_x \psi)_{L^2} =  \int_{\R^d} f(t) \overline{\psi(t-x)}e^{-2\pi i \xi \cdot t} \dt, \qquad (x, \xi) \in \R^{2d}.
\label{STFT-def}
\end{equation}
It holds that  $\|V_\psi f\|_{L^2} = \|\psi \|_{L^2}\|f\|_{L^2}$. Plancherel's theorem implies that
\begin{equation}
V_\psi f(x,\xi) = e^{-2\pi i x \cdot \xi}V_{\widehat{\psi}} \widehat{f}(\xi, -x).
\label{fundamental}
\end{equation}
The adjoint of $V_\psi$ is given by the 
 weak integral
$$
V^\ast_\psi F =\iint_{\R^{2d}} F(x,\xi) M_\xi T_x\psi\: \dx \dxi, \qquad F \in L^2(\R^{2d}).
$$
If  $\gamma \in L^2(\R^d)$ is such that $(\gamma, \psi)_{L^2} \neq 0$, then the reconstruction formula
\begin{equation}
\frac{1}{(\gamma, \psi)_{L^2}} V^\ast_\gamma \circ V_\psi = \operatorname{id}_{L^2(\R^d)}
\label{reconstruction-L2}
\end{equation}
holds. We define the STFT of a function $f$ with respect to a window $\psi$ via the integral at the right-hand side of \eqref{STFT-def} whenever this integral makes sense.

Let $h > 0$. We define the Fr\'echet space
$$
C_{M,h; \operatorname{exp}}(\R^{2d}) = \{ f \in C(\R^{2d}) \, | \, \sup_{(x,\xi) \in \R^{2d}}|f(x,\xi)|e^{-\nu_M(hx) + n|\xi|}< \infty, \quad \forall n \in \N \}
$$
and the $(LB)$-space
$$
C^\circ_{\operatorname{exp}; M,h}(\R^{2d}) = \varinjlim_{n \in \N} C^\circ_{\operatorname{exp},n; M,h}(\R^{2d}),
$$
where $C^\circ_{\operatorname{exp},n; M,h}(\R^{2d})$ denotes the Banach space consisting of all $f \in C(\R^{2d})$ such that
$$
\sup_{(x,\xi) \in \R^{2d}}|f(x,\xi)|e^{-n|x| +\nu_M(h\xi)} < \infty.
$$
\begin{lemma} \label{STFT-mult} \mbox{}
\begin{itemize}
\item[$(i)$] Let $\psi \in \mathcal{U}^{(M)}(\C^d)$. For each $h > 0$ the mapping
$$
V_\psi : \mathcal{P}^{M,h}(\C^d) \rightarrow C_{M,2h; \operatorname{exp}}(\R^{2d})
$$
is continuous.
\item[$(ii)$] Let $\psi \in \mathcal{K}^{(M)}(\R^d)$. For each $h > 0$ the mapping
$$
V_\psi : \mathcal{O}^{M,h}_C(\R^d) \rightarrow C^\circ_{\operatorname{exp}; M,\pi h/\sqrt{d}}(\R^{2d})
$$
is continuous.
\end{itemize}
\end{lemma}
\begin{proof}
In the notation from \cite{D-V} (see also Subsection \ref{weighted}) we have that
$$
\mathcal{P}^{M,h}(\C^d) = \bigcap_{n \in \N} \mathcal{B}^{p!,n}_{e^{-\nu_M(h\,\cdot\,)}}(\R^d), \qquad \mathcal{O}^{M,h}_C(\R^d) = \bigcup_{n \in \N} \mathcal{B}^{M_p,h}_{e^{-n|\,\cdot\,|}}(\R^d).
$$
Hence, both statements follow from \cite[Lemma 4.4]{D-V}.
\end{proof}
\begin{lemma}\label{reconstruction-P}
Let $\psi, \gamma \in \mathcal{U}^{(M)}(\C^d)$ be such that $(\gamma, \psi)_{L^2} \neq 0$. Then,
$$
\int_{\R^d} P(t) \varphi(t) \dt = \frac{1}{(\gamma, \psi)_{L^2}}\iint_{\R^{2d}} V_\psi P(x,\xi) V_{\overline{\gamma}} \varphi(x,-\xi) \dx \dxi, \qquad  \varphi \in \mathcal{U}^{(M)}(\C^d).
$$
\end{lemma}
\begin{proof}
Let $\varphi \in \mathcal{U}^{(M)}(\C^d)$ be arbitrary. In the notation from \cite{D-V} we have that
$$
\mathcal{U}^{(M)}(\C^d) =  \bigcap_{n \in \N} \mathcal{B}^{p!,n}_{e^{\nu_M(n\,\cdot\,)}}(\R^d).
$$ Hence, \cite[Lemma 4.4]{D-V} implies that
\begin{equation}
\sup_{(x,\xi) \in \R^{2d}}|V_{\overline{\gamma}} \varphi(x,\xi)|  e^{\nu_M(nx)+n|\xi|} < \infty, \qquad \forall n \in \N.
\label{STFT-test}
\end{equation}
By \eqref{reconstruction-L2} we have that
$$
\varphi(t) = \frac{1}{(\gamma,\psi)_{L^2}}\iint_{\R^{2d}} V_{\overline{\gamma}} \varphi(x,\xi) M_\xi T_x \overline{\psi}(t) \dx \dxi.
$$
Hence,
\begin{align*}
\int_{\R^d} P(t) \varphi(t) \dt &= \frac{1}{(\gamma,\psi)_{L^2}}  \int_{\R^d} P(t) \left (\iint_{\R^{2d}} V_{\overline{\gamma}} \varphi(x,\xi) M_\xi T_x \overline{\psi}(t) \dx \dxi \right)  \dt \\
&= \frac{1}{(\gamma,\psi)_{L^2}} \iint_{\R^{2d}}  \left ( \int_{\R^d} P(t) M_\xi T_x \overline{\psi}(t) \dt \right)  V_{\overline{\gamma}} \varphi(x,\xi)   \dx \dxi \\
& = \frac{1}{(\gamma, \psi)_{L^2}}\iint_{\R^{2d}} V_\psi P(x,\xi) V_{\overline{\gamma}} \varphi(x,-\xi) \dx \dxi,
\end{align*}
where the switching of the integral is permitted because of Lemma \ref{STFT-mult} and \eqref{STFT-test}.
\end{proof}
\begin{proof}[Proof of Proposition \ref{quantization}]
Fix $\psi, \gamma \in \mathcal{U}^{(M)}(\C^d)$ such that $(\gamma,\psi)_{L^2} = 1$. For $P \in \mathcal{P}^{M,h}(\C^d)$ we define
$$
\langle \widetilde{\mathcal{F}}_h(P), \varphi \rangle = \iint_{\R^{2d}} V_\psi P(x,\xi) V_{\mathcal{F}^{-1}(\overline{\gamma})} \varphi(\xi,x) e^{2\pi i x \cdot \xi} \dx \dxi, \qquad  \varphi \in \mathcal{O}^{M,2\sqrt{d}Hh/\pi}_C(\R^d).
$$
The mapping $\widetilde{\mathcal{F}}_h :  \mathcal{P}^{M,h}(\C^d) \rightarrow \mathcal{O}'^{M,2\sqrt{d}Hh/\pi}_C(\R^d)$ is continuous by Lemma \ref{STFT-mult}. The commutativity of the upper left part of the diagram follows from \eqref{fundamental} and \eqref{reconstruction-L2}; the rest is straightforward to verify. Consequently, \eqref{inc-k-o} is continuous.\\
\indent We now show that $\mathcal{F}(P \cdot {\bf{f}}) = \widetilde{\mathcal{F}}_h(P) \ast \mathcal{F}({\bf{f}})$ for all $P \in \mathcal{P}^{M,h}(\C^d)$ and ${\bf{f}} \in \mathcal{U}'^{(M)}(\C^d;E)$. Let $\varphi \in \mathcal{K}^{(M)}(\R^d)$ be arbitrary. The equation \eqref{fundamental} and Lemma \ref{reconstruction-P} yield that
\begin{align*}
(\widetilde{\mathcal{F}}_h(P) \ast \check{\varphi})\check{\phantom{n}}(t)
&= \iint_{\R^{2d}} V_\psi P(x,\xi) V_{\mathcal{F}^{-1}(\overline{\gamma})} (T_{-t}\varphi)(\xi,x) e^{2\pi i x \cdot \xi} \dx \dxi \\
&= \iint_{\R^{2d}} V_\psi P(x,\xi) V_{\overline{\gamma}} (\mathcal{F}(T_{-t}\varphi))(x,-\xi) \dx \dxi \\
&= \int_{\R^d} P(u) \mathcal{F}(T_{-t}\varphi)(u) \du = \int_{\R^d} P(u) \widehat{\varphi}(u) e^{2\pi i t \cdot u} \du
= \mathcal{F}^{-1}( P \cdot \widehat{\varphi})(t).
\end{align*}
Hence,
$$
\langle  \widetilde{\mathcal{F}}_h(P) \ast \mathcal{F}({\bf{f}}), \varphi \rangle = \langle \mathcal{F}({\bf{f}}), (\widetilde{\mathcal{F}}_h(P) \ast \check{\varphi})\check{\phantom{n}} \rangle  = \langle \mathcal{F}({\bf{f}}), \mathcal{F}^{-1}( P \cdot \widehat{\varphi}) \rangle = \langle \mathcal{F}(P \cdot {\bf{f}}), \varphi \rangle. $$
\end{proof}

\begin{remark}
Because of Proposition \ref{quantization}, from now on we will not distinguish the maps $\widetilde{\mathcal{F}}_h$ for different $h$ and simply denote them by $\widetilde{\mathcal{F}}$.
\end{remark}

Let $h > 0$. We may view $\mathcal{Q}^{M,h}_{{\boldsymbol{\kappa}}}(\R^d;E)$ as a subspace of $\mathcal{K}'^{(M)}(\R^d;E)$ by identifying an element ${\bf{f}} \in \mathcal{Q}^{M,h}_{{\boldsymbol{\kappa}}}(\R^d;E)$ with the operator
$$
\langle {\mathbf{f}}, \varphi \rangle = \int_{\R^d}{ \mathbf{f}}(x) \varphi(x) \dx, \qquad \varphi \in \mathcal{K}^{(M)}(\R^d).
$$
We now discuss the action of convolution operators on $\mathcal{Q}^{M,h}_{{\boldsymbol{\kappa}}}(\R^d;E)$. We need the following structural result.
 \begin{lemma}\label{structural} Let $h > 0$. For each $g \in \mathcal{O}'^{M,h}_C(\R^d)$ there is a family $\{g_\alpha \in C_{\operatorname{exp}}(\R^d) \, | \, \alpha \in \N^d \}$ such that
 $$
 \sum_{\alpha \in \N^d} \frac{M_\alpha \| g_\alpha e^{n|\, \cdot \,|}\|_{L^\infty}}{(2Hh)^{|\alpha|}} < \infty, \qquad \forall n \in \N,
 $$
 and
 $$
 g = \sum_{\alpha \in \N^d} \partial^{\alpha}g_\alpha  \mbox{ on $\mathcal{K}^{(M)}(\R^d)$}.
 $$
 \end{lemma}
\begin{proof}
It is enough to show that for each $g \in \mathcal{O}'^{M,h}_C(\R^d)$ there is a family  $\{g_\alpha  \, | \, \alpha \in \N^d \}$ of measurable functions  such that
$$
\sum_{\alpha \in \N^d} \frac{M_\alpha \| g_\alpha e^{n|\, \cdot \,|}\|_{L^1}}{(2h)^{|\alpha|}} < \infty, \qquad \forall n \in \N,
$$
and $g = \sum_{\alpha \in \N^d} \partial^{\alpha}g_\alpha$ on $\mathcal{K}^{(M)}(\R^d)$. Indeed, let $L\in C(\R^d)\cap C^{\infty}(\R^d\backslash\{0\})$ be a fundamental solution of $\Delta^d$, where $\Delta=\partial^2_1+\ldots+\partial^2_d$ is the Laplacian. Pick $\psi\in\DD(\R^d)$ such that $\psi=1$ on a neighborhood of $0$. Then $\delta-\Delta^d(L\psi)=\chi\in\DD(\R^d)$. Let $\{g_\alpha  \, | \, \alpha \in \N^d \}$ be as above and $\Delta^d=\sum_{\beta}c_{\beta} \partial^{\beta}$; only finitely many $c_{\beta}$ are nonzero. We define
\beqs
\widetilde{g}_{\alpha}=\chi*g_{\alpha}+ \sum_{\beta\leq\alpha}c_{\beta}(L\psi)*g_{\alpha-\beta},\quad \alpha\in\N^d.
\eeqs
Then, the family   $\{\widetilde{g}_{\alpha}  \, | \, \alpha \in \N^d \}$ satisfies all requirements.\\
\indent We now show the above statement. We define the $(LB)$-space
$$
\mathcal{O}^{M,h+}_C(\R^d) =  \varinjlim_{k > h} \mathcal{O}^{M,k}_C(\R^d).
$$
It is straightforward to check that for $h_2>h_1$ and $n_1>n_2$, the inclusion mapping $\mathcal{O}^{M,h_2,n_2}_C(\R^d)\rightarrow \mathcal{O}^{M,h_1,n_1}_C(\R^d)$ is compact; consequently, $\mathcal{O}^{M,h+}_C(\R^d)$ is a $(DFS)$-space. Note that $\mathcal{K}^{(M)}(\R^d) \subset \mathcal{O}^{M,h +}_C(\R^d) \subset \mathcal{O}^{M,h}_C(\R^d)$ with continuous inclusions.
We define $X$ as the Fr\'echet space consisting of all multi-indexed sequences $(g_\alpha)_{\alpha \in \N^d}$ of (equivalence classes) of measurable functions on $\R^d$ such that
$$
\sum_{\alpha \in \N^d} \frac{M_\alpha \| g_\alpha e^{n|\, \cdot \,|}\|_{L^1}}{k^{|\alpha|}} < \infty, \qquad \forall n \in \N, \forall k > h.
$$
For each $(g_{\alpha})_{\alpha\in\N^d}$ and $\alpha\in\N^d$, the linear functional
\beqs
I_{\alpha}((g_{\alpha})_{\alpha\in\N^d}):\mathcal{O}^{M,h +}_C(\R^d) \rightarrow\C,\quad
\langle I_{\alpha}((g_{\alpha})_{\alpha\in\N^d}),\varphi\rangle= (-1)^{|\alpha|}\int_{\R^d}g_{\alpha}(x)\partial^{\alpha}\varphi(x)\dx
\eeqs
is continuous. Furthermore, the linear mapping
$$
S: X \rightarrow \mathcal{O}'^{M,h + }_C(\R^d) :  (g_\alpha)_{\alpha \in \N^d} \mapsto  \sum_{\alpha \in \N^d}I_{\alpha}((g_{\alpha})_{\alpha\in\N^d})
$$
is well defined and continuous as well since the right hand side is absolutely summable in $\mathcal{O}'^{M,h + }_C(\R^d)$. Notice that for each $(g_{\alpha})_{\alpha\in\N^d}$, the restriction of $S((g_{\alpha})_{\alpha\in\N^d})$ to $\mathcal{K}^{(M)}(\R^d)$ is $\sum_{\alpha\in\N^d}\partial^{\alpha}g_{\alpha}$ (the latter is absolutely summable in $\mathcal{K}'^{(M)}(\R^d)$). Thus, it is enough to prove $S$ is surjective. By \cite[Theorem 37.2]{Treves}, this is equivalent to proving that the  transpose of $S$ is injective and has  weak-$*$ closed range. The dual of $X$ coincides with the space $Y$ consisting of all multi-indexed sequences $(f_\alpha)_{\alpha \in \N^d}$ of (equivalence classes) of measurable  functions on $\R^d$ such that for some $n \in \N$ and $k > h$
$$
\sup_{\alpha \in \N^d} \frac{k^{|\alpha|} \| f_\alpha e^{-n|\, \cdot \,|}\|_{L^\infty}}{M_\alpha} < \infty,
$$
under the dual pairing
$$
\langle (f_\alpha)_{\alpha \in \N^d}, (g_\alpha)_{\alpha \in \N^d} \rangle = \sum_{\alpha \in \N^d} \int_{\R^d} f_\alpha(x) g_\alpha(x) \dx, \qquad (f_\alpha)_{\alpha \in \N^d} \in Y, (g_\alpha)_{\alpha \in \N^d} \in X.
$$
As $\mathcal{O}^{M,h + }_C(\R^d)$ is reflexive, the transpose of $S$ is given by
$$
\mathcal{O}^{M,h + }_C(\R^d) \rightarrow Y :  \varphi \mapsto  ((-1)^{|\alpha|} \partial^{\alpha}\varphi)_{\alpha \in \N^d}.
$$
It is clear that this mapping is injective and  has weak-$*$ closed range.
\end{proof}

 \begin{proposition} \label{vv-conv} Let $h > 0$ and let $g \in \mathcal{O}'^{M,h/(2H^2)}_C(\R^d)$. The mapping
 $$
\mathcal{Q}^{M,h}_{{\boldsymbol{\kappa}}}(\R^d;E)  \rightarrow \mathcal{Q}^{M,h/H}_{{\boldsymbol{\kappa}}}(\R^d;E): {\bf{f}} \mapsto g \ast {\bf{f}}
 $$
 is continuous. Moreover, if $\{g_\alpha \in C_{\operatorname{exp}}(\R^d) \, | \, \alpha \in \N^d \}$ leads to a representation of $g$ as in Lemma \ref{structural} (with $h/(2H^2)$ instead of $h$), then
\beq \label{ultrapolsub}
 (g \ast {\bf{f}})(x) = \sum_{\alpha \in \N^d} \int_{\R^d} g_\alpha(t) \partial^{\alpha}{\bf{f}}(x-t) \dt, \qquad {\bf{f}} \in  \mathcal{Q}^{M,h}_{{\boldsymbol{\kappa}}}(\R^d;E).
\eeq
 \end{proposition}
 \begin{proof}
 We start by showing that the mapping
 $$
 S:  \mathcal{Q}^{M,h}_{{\boldsymbol{\kappa}}}(\R^d;E)  \rightarrow \mathcal{Q}^{M,h/H}_{{\boldsymbol{\kappa}}}(\R^d;E)  : {\bf{f}} \mapsto  \sum_{\alpha \in \N^d} \int_{\R^d} g_\alpha(t) T_{t}(\partial^{\alpha}{\bf{f}}) \dt
 $$
 is continuous.  Let ${\bf{f}} \in \mathcal{Q}^{M,h}_{{\boldsymbol{\kappa}}}(\R^d;E) $ be arbitrary. For all $p \in \operatorname{csn}(E)$ we have that
 \begin{align*}
 \sum_{\alpha \in \N^d} p_{\bm{\kappa},h/H} \left( \int_{\R^d} g_\alpha(t) T_{t}(\partial^{\alpha}{\bf{f}}) \dt \right) &\leq   \sum_{\alpha \in \N^d}  \int_{\R^d} |g_\alpha(t)|  p_{\bm{\kappa},h/H} (T_{t}(\partial^{\alpha}{\bf{f}})) \dt \\
 &\leq C_0 \left(  \sum_{\alpha \in \N^d} \frac{M_\alpha \| g_\alpha e^{\kappa_p|\, \cdot \,|}\|_{L^1}}{(h/H)^{|\alpha|}}\right)  p_{\bm{\kappa},h}( {\bf{f}}).
 \end{align*}
This shows that the series $\sum_{\alpha \in \N^d} \int_{\R^d} g_\alpha(t) T_{t}(\partial^{\alpha}{\bf{f}}) \dt$ is absolutely summable in  $\mathcal{Q}^{M,h/H}_{{\boldsymbol{\kappa}}}(\R^d;E)$. Since $\mathcal{Q}^{M,h/H}_{{\boldsymbol{\kappa}}}(\R^d;E)$ is sequentially complete (as $E$ is so), we obtain that $$\sum_{\alpha \in \N^d} \int_{\R^d} g_\alpha(t) T_{t}(\partial^{\alpha}{\bf{f}}) \dt \in \mathcal{Q}^{M,h/H}_{{\boldsymbol{\kappa}}}(\R^d;E).$$ Moreover, the above estimate implies that the mapping $S$ is continuous.  Hence, the result follows by observing that
 $
 g \ast {\bf{f}} = \sum_{\alpha \in \N^d} \int_{\R^d} g_\alpha(t) T_{t}(\partial^{\alpha}{\bf{f}}) \dt
 $
for all ${\mathbf{f}} \in  \mathcal{Q}^{M,h}_{{\boldsymbol{\kappa}}}(\R^d;E)$ because obviously the equality holds as elements of  $\mathcal{K}'^{(M)}(\R^d;E)$.
\end{proof}

\section{Proof of the factorization theorem}\label{thesectiononthefact}
This section is devoted to the proof of Theorem \ref{factorization}. We start with the construction of elements in $\mathcal{P}^{M,1}(\C^d)$ that satisfy suitable lower bounds. We need the following lemma.
\begin{lemma} \label{improvement}  Let $M$ be a weight sequence satisfying $(M.2)^*$. There is a non-decreasing continuous function $\nu: [0,\infty) \rightarrow [0,\infty)$ such that
\begin{itemize}
\item[$(i)$] $\nu_M \asymp \nu$, i.e., $\nu(t) = O(\nu_M(t))$ and $\nu_M(t) = O(\nu(t))$.
\item[$(ii)$] There is a continuous function $\eta: [0,\infty) \rightarrow [0,\infty)$ with $\eta(t) = o(\nu(t))$ and $N \in \N$ such that
$$
|\nu(t_1) - \nu(t_2)| \leq \eta(t_2) (1+ |t_1-t_2|)^N, \qquad  t_1,t_2 \geq 0.
$$
\end{itemize}
\end{lemma}
\begin{proof}
Condition $(M.2)^*$ implies that $\nu_M(2t) = O(\nu_M(t))$. Hence, \cite[Theorem 2.11 and Corollary 2.14]{JGSS} imply that there are $N \in \Z_+$ and $C, L > 0$ such that
$$
\int_1^\infty \frac{\nu_M(ts)}{s^{1+N}} \mathrm{d}s  \leq  L\nu_M(t) + \log C,  \qquad t \geq 0.
$$
Set
$$
\nu(t) = \int_1^\infty \frac{\nu_M(ts)}{s^{1+N}} \mathrm{d}s = t^N \int_t^\infty \frac{\nu_M(s)}{s^{1+N}} \mathrm{d}s, \qquad t \geq 0.
$$
It is clear that $\nu$ is continuous and non-decreasing, and that $\nu_M \asymp \nu$.
Let us show that $\nu$ satisfies property $(ii)$. Set
$$
\eta(t) = \left( \sum_{j = 0}^{N-1} \binom{N}{j} t^j \right)  \int_t^\infty \frac{\nu_M(s)}{s^{1+N}} \mathrm{d}s, \qquad t \geq 0.
$$
It holds that $\eta$ is continuous and that  $\eta(t) = o(\nu(t))$. Let $t_1, t_2 \geq 0$ be arbitrary. If $t_1 \geq t_2$,  then
\begin{align*}
&|\nu(t_1) - \nu(t_2)|= {t_1}^N \int_{t_1}^\infty \frac{\nu_M(s)}{s^{1+N}} \mathrm{d}s - {t_2}^N \int_{t_2}^\infty \frac{\nu_M(s)}{s^{1+N}} \mathrm{d}s \\
&\leq  
 \left( \sum_{j = 0}^{N-1} \binom{N}{j}(t_1-t_2)^{N-j} t_2^j \right) \int_{t_2}^\infty \frac{\nu_M(s)}{s^{1+N}} \mathrm{d}s
\leq
 \eta(t_2) (1+ |t_1-t_2|)^N.
\end{align*}
If $t_2 \geq t_1$,  then
\begin{align*}
&|\nu(t_1) - \nu(t_2)|  = {t_2}^N \int_{t_2}^\infty \frac{\nu_M(s)}{s^{1+N}} \mathrm{d}s - {t_1}^N \int_{t_1}^\infty \frac{\nu_M(s)}{s^{1+N}} \mathrm{d}s \\
&\leq 
\left( \sum_{j = 0}^{N-1} \binom{N}{j}(t_2-t_1)^{N-j} t_1^j \right) \int_{t_2}^\infty \frac{\nu_M(s)}{s^{1+N}} \mathrm{d}s
\leq
\eta(t_2) (1+ |t_1-t_2|)^N.
\end{align*}
\end{proof}
\begin{proposition}\label{par-2}
 Let $M$ be a weight sequence satisfying $(M.2)$ and $(M.2)^*$. There exist  $P \in \mathcal{O}(\C^d)$ and $\delta > 0$ such that for every $n \in \N$ there is $C_n > 0$ such that
$$
C_n^{-1}e^{\nu_M(\delta \Ree z)} \leq |P(z)| \leq C_n e^{\nu_M(\Ree z)}, \qquad   z \in V_n.
$$
\end{proposition}
\begin{proof}
Choose $\nu$, $\eta$ and $N$ as in Lemma \ref{improvement}. We extend $\nu$  to $\R^d$ as the radial function $\nu(x) = \nu(|x|)$, $x \in \R^d$. We set
$$
\widetilde{\nu}(z) = \frac{1}{\pi^{d/2}} \int_{\R^d} \nu(x) e^{-(z-x)\cdot (z-x)} \dx, \qquad z \in \C^d.
$$
Then, $\widetilde{\nu} \in \mathcal{O}(\C^d)$. For all $z \in \C^d$ it holds that
\begin{equation}
\label{ineq-1-2}
|\widetilde{\nu}(z) - \nu(\Ree z)| = \frac{1}{\pi^{d/2}} \left | \int_{\R^d}(\nu(x) - \nu(\Ree z))  e^{-(z-x)\cdot (z-x)} \dx \right |
\leq  Ae^{|\Imm z|^2} \eta(|\Ree z|),
\end{equation}
where $A = \pi^{-d/2} \int_{\R^d} (1+|x|)^Ne^{-|x|^2} \dx$.
Set $\widetilde{P}(z) = e^{\widetilde{\nu}(z)}$, $z \in \C^d$. Then, $\widetilde{P} \in \mathcal{O}(\C^d)$. Let $n \in \N$ be arbitrary. Choose $ B_n > 0$ such that
$$
Ae^{n^2}\eta(t) \leq \frac{1}{2}\nu(t) + \log B_n, \qquad t \geq 0.
$$
Using $\eqref{ineq-1-2}$, we have that
$$
|\widetilde{P}(z)| \leq e^{\nu(\Ree z) +  Ae^{n^2} \eta(|\Ree z|)} \leq B_ne^{2\nu(\Ree z)}, \qquad  z \in V_n,
$$
and
$$
|\widetilde{P}(z)| \geq e^{\nu(\Ree z)  -  Ae^{n^2} \eta(|\Ree z|) } \geq B_n^{-1}e^{\nu(\Ree z)/2}, \qquad  z \in V_n.
$$
Let $L > 1$ be such that
$$
L^{-1}\nu_M(t) - \log L \leq \nu(t) \leq L\nu_M(t) + \log L, \qquad t \geq 0,
$$
and let $K > 1$ be such that
$$
2L \nu_M(t) \leq \nu_M(Kt) + \log K, \qquad t \geq 0.
$$
Set $P(z) = \widetilde{P}(z/K)$, $z \in \C^d$.  Then, $P \in \mathcal{O}(\C^d)$ satisfies the required bounds (with $\delta = 1/K^2$).
\end{proof}
We now combine the theory developed in Section \ref{Fmultipliers} with the function constructed in Proposition \ref{par-2} to obtain a parametrix type identity that shall be used to show Theorem \ref{factorization}.
\begin{lemma}\label{vsrtln159}
Let $E$ be a sequentially complete lcHs and let $\boldsymbol{\kappa} = (\kappa_p)_{p \in \operatorname{csn}(E)}$ be a net of positive numbers.  Let $M$ be a weight sequence satisfying $(M.2)$ and $(M.2)^*$, and let $h > 0$. Let $\mathbf{B}$ be a bounded subset of $\mathcal{Q}^{M,h}_{{\boldsymbol{\kappa}}}(\R^d;E)$. There exist $g \in  \mathcal{O}'^{M,h/(2H^2)}_C(\R^d)$ and $\psi \in \mathcal{K}^{\{M\}}(\R^d)$ such that $g \ast \mathbf{B}$ is a bounded subset of $\mathcal{Q}^{M,h/H}_{{\boldsymbol{\kappa}}}(\R^d;E)$ and
$$
\psi \ast ( g \ast {{\bf{f}}}) = {{\bf{f}}}, \qquad  {{\bf{f}}} \in \mathbf{B}.
$$
\end{lemma}
\begin{proof}
Let $P$ be the function from Proposition \ref{par-2}.  Set $P_h(z) = P(\pi h z/(4\sqrt{d}H^3))$, $z \in \C^d$. Then, $P_h \in \mathcal{P}^{M,\pi h/(4\sqrt{d}H^3)}(\C^d)$ and $1/P_h \in \mathcal{U}^{\{M\}}(\C^d)$. We define $ g = \widetilde{\mathcal{F}}(P_h) \in \mathcal{O}'^{M, h/(2H^2)}_C(\R^d)$ (Proposition \ref{quantization}) and $\psi = \mathcal{F}(1/P_h) \in \mathcal{K}^{\{M\}}(\R^d)$. By Proposition \ref{quantization}, we have that for all $ {\bf{f}}  \in \mathcal{K}'^{(M)}(\R^d;E)$
\begin{align*}
&\psi \ast ( g \ast {\bf{f}}) =  \mathcal{F}(1/P_h) \ast ( \widetilde{\mathcal{F}}(P_h) \ast \mathcal{F}(\mathcal{F}^{-1}( {\bf{f}}))) \\
&= \mathcal{F}(1/P_h) \ast \mathcal{F}( P_h \cdot \mathcal{F}^{-1}( {\bf{f}})) = \mathcal{F} ( (1/P_h) \cdot (P_h \cdot \mathcal{F}^{-1}( {\bf{f}}))) =  {\bf{f}}.
\end{align*}
The result now follows from Proposition \ref{vv-conv}.
\end{proof}
Finally, we shall need the ensuing lemma that allows us to reduce the Beurling case to the Roumieu one.
\begin{lemma}\label{kts135179}
Let $M$ be a weight sequence satisfying  $(M.2)$ and $(M.2)^*$. For every $r \in \mathfrak{R}$ there is $r' \in \mathfrak{R}$ with $r' \preceq r$ such that  $N= (M_p/\prod_{j=0}^p r'_j)_{p \in \N}$ is a weight sequence satisfying
$(M.2)$ and $(M.2)^*$.
\end{lemma}
\begin{proof}
Define $r''_0 = r''_1= 1$ and recursively $r''_j=\min\{r_j, r''_{j-1}m_j/m_{j-1}\}$ for $j \geq 2$. Since $r$ is non-decreasing and $M$ is log-convex, we have that
\beqs
r''_{j+1}=\min\{r_{j+1}, r''_jm_{j+1}/m_j\}\geq \min\{r_j, r''_j\}=r''_j, \qquad j\in\Z_+,
\eeqs
which means that $r''$  is non-decreasing. Next, assume that there is $C\geq 1$ such that $r''_j\leq C$ for all  $j \in \Z_+$. Since $r_j \to \infty$, there is $j_0\geq 2$ such that $r''_j=r''_{j-1}m_j/m_{j-1}$ for all $j\geq j_0$. But then
\beqs
r''_j&=&\frac{m_j}{m_{j-1}}\cdot r''_{j-1}=\frac{m_j}{m_{j-1}}\cdot\frac{m_{j-1}}{m_{j-2}}\cdot r''_{j-2}
=\cdots
=\frac{m_j}{m_{j_0-1}}\cdot r''_{j_0-1}\rightarrow \infty,\,\, \mbox{as}\,\, j\rightarrow \infty,
\eeqs
a contradiction. Hence, $r'' \in\mathfrak{R}$. Moreover, it is clear that $r'' \preceq r$ and that  $m/r''$ is non-decreasing. Define $r'_j=\sqrt{r''_j}$, $j\in \N$. We now show that the sequence $r'$ satisfies all requirements. Note first that $r'\in\mathfrak{R}$ and $r' \preceq r'' \preceq r$. As $m/r''$ and $r''$ are non-decreasing,  $m/r'$ is also non-decreasing. This means that the sequence $N= (M_p/\prod_{j=0}^p r'_j)_{p \in \N}$ is log-convex. It is obvious that $N$  satisfies $(M.2)$. Let us verify that $N$ also satisfies $(M.2)^*$. Choose $p_0,N \in \Z_+$ such that $4m_p \leq m_{Np}$ for all $p \geq p_0$. Since $m/r''$ is non-decreasing, we have that
$$
n_{Np} = \frac{m_{Np}}{r'_{Np}} =  \sqrt{m_{Np}} \sqrt{\frac{m_{Np}}{r''_{Np}}} \geq 2 \sqrt{m_{p}} \sqrt{\frac{m_{p}}{r''_{p}}} = 2n_p, \qquad p \geq p_0.
$$
As $N$ is log-convex, $(M.2)$ and $(M.2)^*$ imply that $\lim_{p \to \infty} N_p^{1/p} = \infty$.
\end{proof}

\begin{proof}[Proof of Theorem \ref{factorization}] A standard argument together with Proposition \ref{char-E*} shows that $\Pi(\mathcal{K}^{[M]}(\R^d))E \subseteq E^{[M]}$ (in the Beurling case, one also needs a two dimensional variant of Lemma \ref{klemma}; see for example \cite[Lemma 2.2.1, p. 18]{PilipovicK}).\\
\indent For the proof of $E^{[M]}\subseteq \Pi(\mathcal{K}^{[M]}(\R^d))E^{[M]}$ and the rest of the theorem it suffices to consider the Roumieu case as, by Lemma \ref{kts135179} and Proposition \ref{char-E*}, the Beurling case follows from it. Suppose that $(\pi,E)$ is a projective (inductive) generalized proto-Banach representation. Let $\widetilde B$ be an arbitrary bounded subset of  $E^{\{M\}}$ and set $\mathbf{B} = \{ \gamma_e \, | \, e \in \widetilde B\}$. By Proposition \ref{char-E*}, there is $h > 0$ ($h > 0$ and $B \in \mathfrak{B}(E)$) such that $\mathbf{B}$ is a bounded subset of $\mathcal{Q}^{M,h}_{{\boldsymbol{\kappa}}}(\R^d;E)$ ($\mathcal{Q}^{M,h}_{\kappa_B}(\R^d;E_{A_B})$).  Lemma \ref{vsrtln159} yields the existence of $g \in  \mathcal{O}'^{M,h/(2H^2)}_C(\R^d)$ and $\psi \in \mathcal{K}^{\{M\}}(\R^d)$ such that $g \ast \mathbf{B}$ is a bounded subset of $\mathcal{Q}^{M,h/H}_{{\boldsymbol{\kappa}}}(\R^d;E)$ ($\mathcal{Q}^{M,h/H}_{\kappa_B}(\R^d;E_{A_B})$) and
$$
\psi \ast ( g \ast  \gamma_e) =  \gamma_e, \qquad  e \in \widetilde B.
$$
Let $\{g_\alpha \in C_{\operatorname{exp}}(\R^d) \, | \, \alpha \in \N^d \} $ be as in Lemma \ref{structural} (with $h/2H^2$ instead of $h$). Set
$$
g \ast e = \sum_{\alpha \in \N^d} \int_{\R^d} g_\alpha(t) \gamma_{e_\alpha}(-t) dt, \qquad  e \in \widetilde B.
$$
Then, $g \ast e \in E$ ($g \ast e \in E_{A_{A_B}}$, where $A_{A_B}\in \mathfrak{B}(E)$ is the set that corresponds to $A_B\in \mathfrak{B}(E)$ in \eqref{inductive}).  By using \eqref{ultrapolsub}, we find that
$
g \ast \gamma_e(x) = \gamma_{g \ast e}(x),$ $e \in \widetilde{B}.$
Hence, Proposition \ref{char-E*} yields that $\widetilde A = \{ g \ast e \, | \, e \in \widetilde B\}$ is a bounded subset of $E^{\{M\}}$ and
$$
\gamma_e(x) =( \psi \ast  \gamma_{g \ast e})(x) = \int_{\R^d} \check{\psi}(t)\gamma_{g \ast e}(x+t) \dt, \qquad e \in \widetilde B.
$$
Evaluating at $x = 0$, we obtain that $e = \Pi(\check{\psi})(g\ast e)$ for all $e \in \widetilde B$.
\end{proof}
\begin{remark}\label{comp-fact}
Let $(\pi,E)$ be a Banach representation. Let $M$ be a weight sequence satisfying $(M.2)$ and $(M.2)^*$. In view of Remark \ref{Banach-top}, a similar argument as in the proof of Theorem \ref{factorization} yields that for every compact set $\widetilde B \subset E^{[M]}$ there is $\psi \in \mathcal{K}^{[M]}(\R^d)$ and a compact set $\widetilde A \subseteq E^{[M]}$ such that $\Pi(\psi)\widetilde A= \widetilde B$.
\end{remark}

\section{The strong factorization property for convolution algebras and modules of ultradifferentiable functions}\label{sect-application}
In this section we employ Theorem \ref{factorization} to establish factorization theorems for concrete instances of convolution algebras and modules of ultradifferentiable functions. We use the symbol $\hookrightarrow$ to denote a dense and continuous inclusion.
We  start by introducing Gelfand-Shilov spaces;  we refer to \cite{ppv} for more information on these spaces. Let $M$ and $A$ be two weight sequences. For $h,k > 0$ we write $\mathcal{S}^{M,h}_{A,k}(\R^d)$ for the Banach space consisting of all $\varphi \in C^\infty(\R^d)$ such that
$$
\| \varphi \|_{\mathcal{S}^{M,h}_{A,k}} :=\sup_{\alpha \in \N^d} \sup_{x \in \R^d} \frac{h^{|\alpha|}|\partial^{\alpha}\varphi(x)|e^{\nu_A(k x)}}{M_{\alpha}} < \infty.
$$
We define
$$
\mathcal{S}^{(M)}_{(A)}(\R^d) = \varprojlim_{h \to \infty} \mathcal{S}^{M,h}_{A,h}(\R^d) \qquad \mbox{and} \qquad \mathcal{S}^{\{M\}}_{\{A\}}(\R^d) = \varinjlim_{h \to 0^+} \mathcal{S}^{M,h}_{A,h}(\R^d).
$$
We denote by $\mathcal{S}'^{[M]}_{[A]}(\R^d)$ the strong dual of $\mathcal{S}^{[M]}_{[A]}(\R^d)$. The spaces $\mathcal{S}^{\{p!^\sigma\}}_{\{p!^\tau\}}(\R^d)$, $\sigma, \tau > 0$, were introduced by Gelfand and Shilov \cite{GelfandShilovVol2}, while the spaces $\mathcal{S}^{\{p!\}}_{\{p!\}}(\R^d)$ and $\mathcal{S}^{(p!)}_{(p!)}(\R^d)$ are the test function spaces for the Fourier hyperfunctions \cite{Kawai} and the Fourier ultrahyperfunctions \cite{PM}, respectively. Note that $\mathcal{K}^{(M)}(\R^d) = \mathcal{S}^{(M)}_{(p!)}(\R^d)$ and $\displaystyle \mathcal{K}^{\{M\}}(\R^d) =  \bigcup_{h > 0} \bigcap_{k >0} \mathcal{S}^{M,h}_{p!,k}(\R^d)$.

\subsection{Translation-invariant Banach spaces} Fix a weight sequence $M$ satisfying $(M.2)$ and $(M.2)^*$, and a weight sequence $A$ satisfying $(M.2)$ and $p! \subset A$.  Consequently, $\mathcal{K}^{[M]}(\R^d) \subseteq \mathcal{S}^{[M]}_{[A]}(\R^d)$.

Following \cite{D-P-P-V,D-P-V}, we call a Banach space $E$ a \emph{translation-invariant Banach space (TIB) of class ${[M]}-{[A]}$} if the following three properties are satisfied:
\begin{itemize}
\item[$(i)$] $\mathcal{S}^{{[M]}}_{{[A]}}(\R^d) \hookrightarrow E \hookrightarrow \mathcal{S}'^{{[M]}}_{{[A]}}(\R^d)$;
\item[$(ii)$] $T_x(E) \subseteq E$ for all $x \in \R^d$;
\item[$(iii)$] There are $\kappa, C > 0$ (for every $\kappa > 0$ there is $C > 0$) such that\footnote{The continuity of the mappings $T_x: E \rightarrow E$, $x \in \R^d$, follows from the closed graph theorem.}
$$
\omega(x) := \| T_x\|_{\mathcal{L}_b(E,E)} \leq Ce^{\nu_A(\kappa x)}, \qquad x \in \R^d.
$$
\end{itemize}
\begin{example}\label{example-TIB}
Let $w$ be a positive measurable function on $\R^d$ such that
$$
w(x + y) \leq Cw(x)e^{\nu_A(\kappa y)}, \qquad x,y \in \R^d,
$$  for some $\kappa, C> 0$ (for every $\kappa > 0$ and some $C=C_{\kappa} > 0$).
For $p \in [1,\infty]$, we denote by $L^p_w(\R^d)$ the space of (equivalence classes) of measurable functions $f$ on $\R^d$ such that $\|f\|_{L^p_w} := \|fw\|_{L^p} < \infty$. Furthermore, $C_{w,0}(\R^d)$ stands for the Banach space of  continuous functions $f$ on $\R^d$ such that
$\lim_{|x| \rightarrow \infty} f(x)w(x) = 0$ endowed with the norm  $\|\: \:\|_{L^\infty_w}$. The spaces $L^p_w(\R^d)$, $1\leq p < \infty$, and $C_{w,0}(\R^d)$ are TIB of class ${[M]}-{[A]}$.
\end{example}

\smallskip

Fix a TIB $E$ of class ${[M]}-{[A]}$.   We have that
$$
T: \R^d \times E \rightarrow E:\ (x,e) \mapsto T_{x}(e)
$$
is continuous. This means that $(T,E)$ is a Banach representation.  One shows \cite[Proposition 3.10]{D-P-V} that the convolution mapping $\ast : \mathcal{S}^{[M]}_{{[A]}}(\R^d) \times \mathcal{S}^{{[M]}}_{{[A]}}(\R^d) \rightarrow \mathcal{S}^{{[M]}}_{{[A]}}(\R^d)$ uniquely extends to a continuous bilinear mapping $\ast: L^1_\omega(\R^d) \times E \rightarrow E$.  By \cite[Lemma 3.7]{D-P-V}  and  the fact that $\mathcal{K}^{[M]}(\R^d)$ is dense in both $C_{\operatorname{exp}}(\R^d)$ and $E$, we have that
$$
\Pi(f)e = f \ast e, \qquad f \in C_{\operatorname{exp}}(\R^d),\, e \in E.
$$

For $h >0$ we write $\mathcal{D}^{M,h}_E$ for the Banach space consisting of all $\varphi \in E$ such that\footnote{$\partial^\alpha \varphi$ stands here for the $\alpha$-th $\mathcal{S}'^{[M]}_{[A]}(\R^d)$-derivative of $\varphi$.} $\partial^\alpha \varphi \in E$  for all $\alpha \in \N^d$ and
$$
\sup_{\alpha \in \N^d} \frac{h^{|\alpha|} \|\partial^\alpha \varphi \|_E}{M_\alpha} < \infty.
$$
We then define
$$
\mathcal{D}^{(M)}_E = \varprojlim_{h \to \infty} \mathcal{D}^{M,h}_E  \qquad \mbox{and}\qquad \mathcal{D}^{\{M\}}_E = \varinjlim_{h \to 0^+} \mathcal{D}^{M,h}_E.
$$
It holds that
$$
\mathcal{S}^{{[M]}}_{[A]}(\R^d) \hookrightarrow \mathcal{D}^{[M]}_E \hookrightarrow E \hookrightarrow \mathcal{S}'^{{[M]}}_{[A]}(\R^d).
$$
Moreover, one can show that the elements $\varphi$ of $\mathcal{D}^{[M]}_E$ are in fact smooth functions and that they satisfy
$$
\sup_{\alpha \in \N^d} \sup_{x\in\R^d}\frac{h^{|\alpha|} |\partial^{\alpha}\varphi(x)|}{M_\alpha \omega(-x)} < \infty
$$
for all $h > 0$ (for some $h > 0$) \cite[Proposition 4.7]{D-P-V}.

We write $\mathcal{D}_{E}$ for the space  consisting of all those $\varphi \in E$ such that $\partial^\alpha \varphi \in E$ for all $\alpha \in \N^d$.
Reasoning as in \cite[Lemma 5.7]{Debrouwere}, it can be shown that $E^\infty = \mathcal{D}_E$ and $\varphi_\alpha = \partial^\alpha \varphi$ for all $\alpha \in \N^d$ and $\varphi \in E^{\infty}$. Hence, $E^{[M]} = \mathcal{D}^{[M]}_E$ as locally convex spaces.

Theorem \ref{factorization} and Remark \ref{comp-fact} yield the following result.

\begin{theorem}\label{factorization-TIB} We have that
$$
\mathcal{K}^{[M]}(\R^d) \ast E = \mathcal{K}^{[M]}(\R^d) \ast \mathcal{D}^{[M]}_E  = \mathcal{D}^{[M]}_E.
$$
Furthermore, for every bounded (compact)  set $B \subset \mathcal{D}^{[M]}_E$ there are $\psi \in \mathcal{K}^{[M]}(\R^d)$ and a bounded (compact)  set $A \subset \mathcal{D}^{[M]}_E$ such that $\psi \ast A = B$.
\end{theorem}
\begin{corollary} \label{factorization-TIB-cor}
We have that
$$
\mathcal{S}^{[M]}_{[A]}(\R^d) \ast E = \mathcal{S}^{[M]}_{[A]}(\R^d) \ast \mathcal{D}^{[M]}_E  = \mathcal{D}^{[M]}_E.
$$
Moreover, for every bounded (compact) set $B \subset \mathcal{D}^{[M]}_E$ there are $\psi \in \mathcal{S}^{[M]}_{[A]}(\R^d)$ and a bounded (compact) set $A \subset \mathcal{D}^{[M]}_E$ such that $\psi \ast A = B$.
\end{corollary}

We shall now apply Corollary \ref{factorization-TIB-cor} to improve a useful convolution average characterization of the dual of $\mathcal{D}^{[M]}_{E}$ obtained in \cite{D-P-P-V,D-P-V}. We start with a brief discussion of the strong dual $E'$ of $E$.
The convolution on $E'$ is defined via transposition:
$$
\langle f \ast e', e \rangle := \langle e', \check{f} \ast e \rangle, \qquad e \in E,
$$
for $f \in L^1_\omega(\R^d)$ and $e' \in E'$.
We set $E'_\ast := L^1_\omega(\R^d) \ast E'$. Since the Beurling algebra $(L^1_\omega(\R^d),\ast)$ admits bounded approximation units, the Cohen-Hewitt factorization theorem implies that $E'_\ast$ is a closed linear subspace of $E'$.
Moreover, we have the following explicit description of $E'_\ast$ \cite[Proposition 3.18]{D-P-V},
$$
E'_\ast = \{ e' \in E' \, | \, \lim_{x \to 0} \| T_x e' - e'\|_{E'} = 0 \}.
$$
In general, $E'$ is not a TIB of class ${[M]}-{[A]}$: The embedding $\mathcal{S}^{{[M]}}_{{[A]}}(\R^d) \rightarrow E'$ needs not be dense and the mappings $\R^d \to E', x \mapsto T_x e'$, $e' \in E'$ fixed, may fail to be continuous; consider, e.g., $E = L^1(\R^d)$. This causes various difficulties and is one of the reason why the space $E'_\ast$ was introduced in \cite{D-P-V}.
If $E$ is reflexive, $E'$ is a TIB of class ${[M]}-{[A]}$ and $E' = E'_\ast$ \cite[Theorem 3.17]{D-P-V}.

We denote by $\mathcal{D}'^{[M]}_{E'_\ast}$ the strong dual of $\mathcal{D}^{[M]}_E$. We have the following continuous inclusions
$$
\mathcal{S}^{{[M]}}_{[A]}(\R^d) \rightarrow E' \rightarrow \mathcal{D}'^{[M]}_{E'_\ast} \rightarrow \mathcal{S}'^{{[M]}}_{[A]}(\R^d).
$$
Moreover, the following characterization of $\mathcal{D}'^{[M]}_{E'_\ast}$ in terms of convolution averages holds \cite[Theorem 4.9]{D-P-V}
$$
\mathcal{D}'^{[M]}_{E'_\ast} = \{ f \in \mathcal{S}'^{[M]}_{[A]}(\R^d) \, | \, f \ast \varphi \in E'_\ast \mbox{ for all } \varphi \in \mathcal{S}^{[M]}_{[A]}(\R^d)\}.
$$
The above equality suggests to embed the space $\mathcal{D}'^{[M]}_{E'_\ast}$ into the space of $E'_\ast$-valued ultradistributions $\mathcal{S}'^{[M]}_{[A]}(\R^d; E'_\ast) = \mathcal{L}_b(\mathcal{S}^{[M]}_{[A]}(\R^d), E'_\ast)$ by regarding the elements of $\mathcal{D}'^{[M]}_{E'_\ast}$ as kernels of convolution operators:
$$
\iota: \mathcal{D}'^{[M]}_{E'_\ast} \rightarrow \mathcal{S}'^{[M]}_{[A]}(\R^d; E'_\ast),\,\,\,\,\langle \iota(f), \varphi \rangle = f \ast \varphi,\,\,\, \varphi \in \mathcal{S}^{[M]}_{[A]}(\R^d).
$$
The mapping $\iota$ is a well-defined continuous inclusion with closed range \cite[p.~168-169]{D-P-V}. We now show that $\iota$ is actually a topological embedding.

\begin{proposition}\label{embedding-TIB}
The mapping $\iota$ is a topological embedding.
\end{proposition}
\begin{proof}
By the above remarks, we only need to show that $\iota$ is open, i.e., that for every neighborhood $U$ of zero in $\mathcal{D}'^{[M]}_{E'_\ast}$ there is a neighborhood $V$ of zero in $\mathcal{S}'^{[M]}_{[A]}(\R^d; E'_\ast)$ such that $V \cap \iota(\mathcal{D}'^{[M]}_{E'_\ast}) \subseteq  \iota(U)$. We may assume that $U$ is the polar set of a bounded set $B$ of $\mathcal{D}^{[M]}_E$, that is,
$$
U = \{ f \in \mathcal{D}'^{[M]}_{E'_\ast} \, | \, \sup_{\varphi \in B} |\langle f, \varphi \rangle| \leq 1 \}.
$$
By Corollary \ref{factorization-TIB-cor}, there exist $\psi \in \mathcal{S}^{[M]}_{[A]}(\R^d)$ and a bounded set $A \subset \mathcal{D}^{[M]}_E$ such that $\psi \ast A = B$. Let $R > 0$ be such  that $\|\varphi\|_E \leq R$ for all $\varphi \in A$. Consider the following neighborhood of zero in $\mathcal{S}'^{[M]}_{[A]}(\R^d; E'_\ast)$
$$
V = \{ \mathbf{f} \in \mathcal{S}'^{[M]}_{[A]}(\R^d; E'_\ast) \, | \, \|\langle\mathbf{f},\check{\psi}\rangle\|_{E'} \leq 1/R \}.
$$
For $f \in  \mathcal{D}'^{[M]}_{E'_\ast}$ with $\iota(f) \in V$ we have that
$$
\sup_{\varphi \in B} |\langle f, \varphi \rangle| = \sup_{\chi \in A} |\langle f, \psi \ast \chi \rangle| =  \sup_{\chi \in A} |\langle f \ast \check \psi, \chi \rangle| =  \sup_{\chi \in A} |\langle \iota(f)(\check \psi), \chi \rangle| \leq 1,
$$
which means that $f \in U$.
\end{proof}

\begin{remark}
The proof of Proposition \ref{embedding-TIB} in fact shows that
 $$
\iota: \mathcal{D}'^{[M]}_{E'_\ast} \rightarrow \mathcal{L}_\sigma(\mathcal{S}^{[M]}_{[A]}(\R^d), E'_\ast)
$$
is a topological embedding.
\end{remark}

\subsection{Weighted spaces of ultradifferentiable functions}\label{weighted}  Fix a weight sequence $M$ satisfying $(M.2)$ and $(M.2)^*$.

A pointwise increasing sequence $\mathcal{W} = (w_n)_{n \in \N}$ of  positive continuous functions on $\R^d$ is called an \emph{increasing weight system} if the following two properties are satisfied:
\begin{equation}
\forall n \in \N \, \exists m \geq n \, : \, \lim_{|x| \to \infty} \frac{w_n(x)}{w_m(x)} = 0
\label{proper}
\end{equation}
and
\begin{equation}
 \forall n \in \N \,  \exists m \geq n \, \exists \kappa > 0 \, \exists C > 0 \, \forall x,y \in \R^d \, : \, w_n(x+ y) \leq Cw_{m}(x)e^{\kappa|y|}.
\label{exp-bdd}
\end{equation}
Similarly, a pointwise decreasing sequence $\mathcal{V} = (v_n)_{n \in \N}$ of  positive continuous functions on $\R^d$ is called a \emph{decreasing weight system} if the following two properties are satisfied:
\begin{equation}
\forall n \in \N \, \exists m \geq n \, : \, \lim_{|x| \to \infty} \frac{v_m(x)}{v_n(x)} = 0,
\label{proper-1}
\end{equation}
\begin{equation}
 \forall n \in \N \, \exists m \geq n \, \exists \kappa > 0 \, \exists C > 0 \, \forall x,y \in \R^d \, : \, v_m(x+ y) \leq Cv_{n}(x)e^{\kappa|y|}.
\label{exp-bdd-1}
\end{equation}
We remark that \eqref{proper-1} is condition $(S)$ from \cite{Bierstedt}.
\begin{example}\label{example-1}
Let $A$ be a weight sequence satisfying $p! \subset A$. Then, $\mathcal{W}_{(A)} = (e^{\nu_A(n \, \cdot \,)})_{n \in \N}$
is an increasing weight systems. Likewise,  $\mathcal{V}_{\{A\}} = (e^{\nu_A( \,\cdot \,/n)})_{n \in \Z_+}$ is a decreasing weight system.
Indeed, for all $h > k > 0$ it holds that $\nu_A(ht) - \nu_A(kt) \rightarrow \infty$ as $t \rightarrow \infty$; this follows from \cite[Equation (3.11)]{Komatsu}. Hence, the conditions \eqref{proper} and \eqref{proper-1} are satisfied. The conditions \eqref{exp-bdd} and \eqref{exp-bdd-1} follow from the assumption $p! \subset A$.
\end{example}
Let $w$ be a positive continuous function on $\R^d$. For $h > 0$  we write $\mathcal{B}^{M,h}_{w}(\R^d)$ for the Banach space consisting of all $\varphi \in C^\infty(\R^d)$ such that
$$
\| \varphi \|_{\mathcal{B}^{M,h}_w} := \sup_{\alpha \in \N^d} \sup_{x \in \R^d} \frac{h^{|\alpha|}|\partial^{\alpha}\varphi(x)|w(x)}{M_\alpha} < \infty.
$$
Given an increasing weight system $\mathcal{W} = (w_n)_{n \in \N}$, we define
$$
\mathcal{B}^{(M)}_{\mathcal{W}}(\R^d) = \varprojlim_{n \in \N} \mathcal{B}^{M,n}_{w_n}(\R^d) \qquad \mbox{and}\qquad \mathcal{B}^{\{M\}}_{\mathcal{W}}(\R^d) =  \varinjlim_{h \rightarrow 0^+} \varprojlim_{n \in \N} \mathcal{B}^{M,h}_{w_n}(\R^d).
$$
Similarly, given a decreasing weight system $\mathcal{V} = (v_n)_{n \in \N}$, we define
$$
\mathcal{B}^{(M)}_{\mathcal{V}}(\R^d) = \varinjlim_{n \in \N} \varprojlim_{h \rightarrow \infty}  \mathcal{B}^{M,h}_{v_n}(\R^d) \qquad \mbox{and}\qquad \mathcal{B}^{\{M\}}_{\mathcal{V}}(\R^d) =  \varinjlim_{n \in \N} \mathcal{B}^{M,1/n}_{v_n}(\R^d).
$$
Furthermore, we set
$$
\mathcal{B}^{M,h}_{\mathcal{W}}(\R^d) = \varprojlim_{n \in \N} \mathcal{B}^{M,h}_{w_n}(\R^d) \qquad \mbox{and}\qquad \mathcal{B}^{(M)}_{v_n}(\R^d) = \varprojlim_{h \to \infty} \mathcal{B}^{M,h}_{v_n}(\R^d).
$$
We refer to \cite{D-V} for more information on these spaces.
\begin{example}\label{example-2}
Let $A$ be a weight sequence satisfying $p! \subset A$. Then,
$$
\mathcal{S}^{(M)}_{(A)}(\R^d) = \mathcal{B}^{(M)}_{\mathcal{W}_{(A)}}(\R^d)\qquad \mbox{and}\qquad \mathcal{S}^{\{M\}}_{\{A\}}(\R^d) = \mathcal{B}^{\{M\}}_{\mathcal{V}_{\{A\}}}(\R^d).
$$
Notice that we have $\mathcal{K}^{\{M\}}(\R^d) = \mathcal{B}^{\{M\}}_{\mathcal{W}_{(p!)}}(\R^d)$.
\end{example}
We fix an increasing weight system $\mathcal{W} = (w_n)_{n \in \N}$ and a decreasing weight system $\mathcal{V} = (v_n)_{n \in \N}$. Given a positive continuous function $w$ on $\R^d$, we denote by $C_w(\R^d)$  the Banach space consisting of all $f \in C(\R^d)$ such that $\|f\|_{L^\infty_w} < \infty$.  We define the Fr\'echet space
$$
\mathcal{W}C(\R^d) = \varprojlim_{n \in \N} C_{w_n}(\R^d)
$$
and the $(LB)$-space
 $$
 \mathcal{V}C(\R^d) = \varinjlim_{n \in \N} C_{v_n}(\R^d).
$$
Since $\mathcal{V}$ satisfies \eqref{proper-1}, the space $\mathcal{V}C(\R^d)$ is  boundedly retractive and thus complete \cite[p.~100]{Bierstedt}. Then, $(T, \mathcal{W}C(\R^d))$ is a projective  generalized proto-Banach representation and  $(T, \mathcal{V}C(\R^d))$ is an inductive generalized proto-Banach representation.
Moreover, we have that
$$
\Pi(f)g = f \ast g, \qquad f \in C_{\operatorname{exp}}(\R^d), \, g \in E,
$$
where $E$ denotes either $\mathcal{W}C(\R^d)$ or $\mathcal{V}C(\R^d)$.
\begin{lemma}\label{uv-weighted} \mbox{}
\begin{itemize}
\item[$(i)$] $\mathcal{W}C(\R^d)^{[M]} = \mathcal{B}^{[M]}_{\mathcal{W}}(\R^d)$ as sets. Moreover, $B \subset \mathcal{W}C(\R^d)^{(M)}$ ($B \subset \mathcal{W}C(\R^d)^{\{M\}}$) is bounded if and only if  $B$ is bounded in $\mathcal{B}^{(M)}_{\mathcal{W}}(\R^d)$ ($B$ is bounded in $\mathcal{B}^{M,h}_{\mathcal{W}}(\R^d)$ for some $h > 0$).
\item[$(ii)$] $\mathcal{V}C(\R^d)^{[M]} = \mathcal{B}^{[M]}_{\mathcal{V}}(\R^d)$ as sets. Moreover, $B \subset \mathcal{V}C(\R^d)^{(M)}$ ($B \subset \mathcal{V}C(\R^d)^{\{M\}}$) is bounded if and only if  $B$ is bounded in $\mathcal{B}^{(M)}_{v_n}(\R^d)$ for some $n \in \N$ ($B$ is bounded in $\mathcal{B}^{\{M\}}_{\mathcal{W}}(\R^d)$).
\end{itemize}
\end{lemma}
\begin{proof}
Let $E$ denote either $\mathcal{W}C(\R^d)$ or $\mathcal{V}C(\R^d)$.  We write $\mathcal{D}_{E}$ for the space  consisting of all $\varphi \in C^\infty(\R^d)$ such that $\partial^\alpha \varphi \in E$ for all $\alpha \in \N^d$. Then,   $E^\infty = \mathcal{D}_E$ and $\varphi_\alpha = (-1)^{|\alpha|}\partial^\alpha \varphi$ for all $\alpha \in \N^d$ and $\varphi \in E^{\infty}$. Hence, the result follows from Proposition \ref{char-E*} (and Remark \ref{char-E*-Frechet}).
\end{proof}
Theorem \ref{factorization} and Lemma \ref{uv-weighted} imply the ensuing factorization result.
\begin{theorem}\label{factorization-GS} \mbox{}
\begin{itemize}
\item[$(i)$] Let $\mathcal{W} = (w_n)_{n \in \N}$ be an increasing weight system. We have that
$$
\mathcal{K}^{[M]}(\R^d) \ast \mathcal{B}^{{[M]}}_\mathcal{W}(\R^d)  = \mathcal{B}^{{[M]}}_{\mathcal{W}}(\R^d).
$$
Moreover, in the Beurling case, for every bounded set $B \subset \mathcal{B}^{(M)}_{\mathcal{W}}(\R^d)$ there are $\psi \in \mathcal{K}^{(M)}(\R^d)$ and a bounded set $A \subset \mathcal{B}^{(M)}_{\mathcal{W}}(\R^d)$ such that $\psi \ast A = B$. In the Roumieu case, for every $h > 0$ and every bounded set $B \subset \mathcal{B}^{M,h}_{\mathcal{W}}(\R^d)$ there are $\psi \in \mathcal{K}^{\{M\}}(\R^d)$ and a bounded set $A \subset \mathcal{B}^{M,k}_{\mathcal{W}}(\R^d)$, for some $k > 0$, such that $\psi \ast A = B$.
\item[$(ii)$] Let $\mathcal{V} = (v_n)_{n \in \N}$ be a  decreasing weight system. We have that
$$
\mathcal{K}^{[M]}(\R^d) \ast \mathcal{B}^{{[M]}}_\mathcal{V}(\R^d)  = \mathcal{B}^{{[M]}}_{\mathcal{V}}(\R^d).
$$
Moreover, in the Beurling case, for every $n \in \N$ and every bounded set $B \subset \mathcal{B}^{(M)}_{v_n}(\R^d)$ there are $\psi \in \mathcal{K}^{(M)}(\R^d)$ and a bounded set $A \subset \mathcal{B}^{(M)}_{v_m}(\R^d)$, for some $m \in \N$, such that $\psi \ast A = B$. In the Roumieu case,  for every every bounded set $B \subset \mathcal{B}^{\{M\}}_{\mathcal{V}}(\R^d)$ there are $\psi \in \mathcal{K}^{\{M\}}(\R^d)$ and a bounded set $A \subset \mathcal{B}^{\{M\}}_{\mathcal{V}}(\R^d)$ such that $\psi \ast A = B$.
\end{itemize}
\end{theorem}
\begin{corollary} \label{cor-GS}
Let $A$ be a weight sequence satisfying $p! \subset A$. We have that
$$
\mathcal{K}^{{[M]}}(\R^d) \ast\mathcal{S}^{{[M]}}_{{[A]}}(\R^d) = \mathcal{S}^{{[M]}}_{{[A]}}(\R^d) \ast\mathcal{S}^{{[M]}}_{{[A]}}(\R^d) =\mathcal{S}^{{[M]}}_{{[A]}}(\R^d).
$$
Moreover, for every bounded set $B \subset \mathcal{S}^{{[M]}}_{{[A]}}(\R^d)$ there are $\psi  \in \mathcal{K}^{[M]}(\R^d)$ and a bounded set $A \subset \mathcal{S}^{{[M]}}_{{[A]}}(\R^d)$ such that $\psi \ast A = B$.
\end{corollary}
\begin{remark}\label{regular}
Theorem \ref{factorization-GS} implies that we can factorize bounded sets in $\mathcal{B}^{\{M\}}_{\mathcal{W}}(\R^d)$ and $\mathcal{B}^{(M)}_{\mathcal{V}}(\R^d)$  provided that these $(LF)$-spaces are regular.
The space $\mathcal{B}^{\{M\}}_{\mathcal{W}}(\R^d)$ is regular if $\mathcal{W}$ satisfies
\begin{equation}
\exists n \in \N \, \forall m \geq n \, \exists k \geq m \, \exists C > 0 \, \forall x \in \R^d \, : \, w^2_m(x) \leq C w_n(x)w_k(x).
\label{DN}
\end{equation}
If $M$ additionally satisfies Komatsu's condition $(M.3)$ (strong non-quasianalyticity) \cite{Komatsu}, condition \eqref{DN} is also necessary for $\mathcal{B}^{(M)}_{\mathcal{V}}(\R^d)$ to be regular \cite[Theorem 3.6]{D-V}. Likewise, the space $\mathcal{B}^{(M)}_{\mathcal{V}}(\R^d)$ is regular if $\mathcal{V}$ satisfies
\begin{equation}
\forall n \in \N \, \exists m \geq n \, \forall k \geq m \, \exists \theta \in (0,1) \, \exists C > 0 \, \forall x \in \R^d \, : \,
 v_m(x) \leq C v^{1-\theta}_n(x)v^{\theta}_k(x).
\label{Omega}
\end{equation}
If in addition $M$  satisfies $(M.3)$,  condition \eqref{Omega} is also necessary for  $\mathcal{B}^{(M)}_{\mathcal{V}}(\R^d)$ to be regular \cite[Theorem 3.7]{D-V}. The conditions \eqref{DN} and \eqref{Omega} are closely connected and inspired by the linear topological invariants $(DN)$ and $(\Omega)$ for Fr\'echet spaces \cite{MeiseVogtBook}.
\end{remark}
Theorem \ref{factorization} and Remark \ref{regular} then yield the following result.
\begin{corollary} \label{cor-K}
We have that
$
\mathcal{K}^{{[M]}}(\R^d) \ast \mathcal{K}^{{[M]}}(\R^d) = \mathcal{K}^{{[M]}}(\R^d).
$
Moreover, for every bounded set $B \subset \mathcal{K}^{{[M]}}(\R^d)$ there are $\psi  \in \mathcal{K}^{[M]}(\R^d)$ and a bounded set $A \subset \mathcal{K}^{{[M]}}(\R^d)$ such that $\psi \ast A = B$.
\end{corollary}

\end{document}